\documentclass{article}

\usepackage[utf8]{inputenc}
\usepackage{amsmath}
\usepackage{amsfonts}
\usepackage{amssymb}
\usepackage{amsthm}
\usepackage{latexsym}
\usepackage{graphicx}
\usepackage{xcolor}
\usepackage{epsfig}
\usepackage{amsbsy}
\usepackage{textcomp}
\usepackage{comment}
\usepackage[hypcap]{caption}
\usepackage{hyperref}
\usepackage{aliascnt}

\theoremstyle{plain}
\newtheorem{thm}{Theorem}[section]
\newaliascnt{lem}{thm}
\newtheorem{lem}[lem]{Lemma}
\aliascntresetthe{lem}
\newaliascnt{pro}{thm}
\newtheorem{pro}[pro]{Proposition}
\aliascntresetthe{pro}
\newaliascnt{cor}{thm}
\newtheorem{cor}[cor]{Corollary}
\aliascntresetthe{cor}
\newaliascnt{que}{thm}
\newtheorem{que}[que]{Question}
\aliascntresetthe{que}
\newaliascnt{rem}{thm}

\aliascntresetthe{rem}
\newaliascnt{con}{thm}
\newtheorem{con}[con]{Conjecture}
\aliascntresetthe{con}
\newaliascnt{clm}{thm}

\aliascntresetthe{clm}
\newtheorem*{clmnn}{Claim}

\newtheorem{thm0}{Theorem}

\theoremstyle{definition}
\newaliascnt{defi}{thm}
\newtheorem{defi}[defi]{Definition}
\aliascntresetthe{defi}
\newaliascnt{exm}{thm}
\newtheorem{exm}[exm]{Example}
\aliascntresetthe{exm}

\newcommand{\mcm}[3]{\newcommand{#1}[#2]{{\ensuremath{#3}}}} 

\mcm{\tuple}{1}{\langle #1 \rangle}
\mcm{\name}{1}{\ulcorner #1 \urcorner}
\mcm{\Fbb}{0}{\mathbb{F}}
\mcm{\Nbb}{0}{\mathbb{N}}
\mcm{\Zbb}{0}{\mathbb{Z}}
\mcm{\Rbb}{0}{\mathbb{R}}
\mcm{\Cbb}{0}{\mathbb{C}}
\mcm{\Acal}{0}{\cal A}
\mcm{\Bcal}{0}{\cal B}
\mcm{\Ccal}{0}{\cal C}
\mcm{\Dcal}{0}{\cal D}
\mcm{\Ecal}{0}{\cal E}
\mcm{\Fcal}{0}{\cal F}
\mcm{\Gcal}{0}{\cal G}
\mcm{\Hcal}{0}{\cal H}
\mcm{\Ical}{0}{\cal I}
\mcm{\Lcal}{0}{\cal L}
\mcm{\Mcal}{0}{\cal M}
\mcm{\Wcal}{0}{\cal W}
\mcm{\Ncal}{0}{\cal N}
\mcm{\Pcal}{0}{{\cal P}}
\mcm{\Qcal}{0}{{\cal Q}}
\mcm{\Rcal}{0}{{\cal R}}
\mcm{\Scal}{0}{{\cal S}}
\mcm{\Tcal}{0}{{\cal T}}
\mcm{\Ucal}{0}{{\cal U}}
\mcm{\Vcal}{0}{{\cal V}}
\mcm{\Mfrak}{0}{\mathfrak M}

\mcm{\restric}{0}{\upharpoonright}
\mcm{\upset}{0}{\uparrow}
\mcm{\onto}{0}{\twoheadrightarrow}
\mcm{\smallNbb}{0}{{\small \mathbb{N}}}
\DeclareMathOperator{\preop}{op}
\mcm{\op}{0}{^{\preop}}

\renewcommand{\-}{\setminus}
\newcommand{\Ter}{\mbox{\rm Ter}}
\newcommand{\ter}{\mbox{\rm Ter}}
\newcommand{\Ini}{\mbox{\rm Ini}}
\newcommand{\ini}{\mbox{\rm Ini}}
\newcommand{\cev}[1]{\reflectbox{\ensuremath{\vec{\reflectbox{\ensuremath{#1}}}}}}

\begin{document}
\author{S. Hadi Afzali Borujeni \and Hiu Fai Law \and Malte Müller}
\title{Infinite Gammoids: Minors and Duality}
\date{\today}
\maketitle
\setcounter{section}{-1}

\begin{abstract}
This sequel to \cite{ALM} considers minors and duals of infinite gammoids. We prove that a class of gammoids definable by digraphs not containing a certain type of substructure, called an outgoing comb, is minor-closed. Also, we prove that finite-rank minors of gammoids are gammoids. 
Furthermore, the topological gammoids \cite{Car14+} are proved to coincide, as matroids, with the finitary gammoids. A corollary is that topological gammoids are minor-closed. 

It is a well-known fact that the dual of any finite strict gammoid is a transversal matroid. 
The class of alternating-comb-free strict gammoids, introduced in \cite{ALM}, contains examples which are not dual to any transversal matroid.
However, we describe the duals of matroids in this class as a natural extension of transversal matroids. 
While finite gammoids are closed under duality,
we construct a strict gammoid that is not dual to any gammoid. 
\end{abstract}

\section{Introduction}
Infinite matroid theory has recently seen a surge of activity (e.g. \cite{AB12}, \cite{BC12}, \cite{BW12}), after Bruhn et al \cite{BDKPW} found axiomatizations of infinite matroids with duality, solving a long-standing problem of Rado \cite{Rad66}. As part of this re-launch of the subject, we initiated an investigation of infinite gammoids in \cite{ALM}.

A {\em dimaze} is a digraph equipped with a specific set of sinks, the (set of) {\em exits}. 
A dimaze {\em contains} another dimaze, if, in addition to digraph containment, the exits of the former include those of the latter. 
In the context of digraphs, any path or ray (i.e.~one-way infinite path) is forward oriented. 
A {\em linking fan} is a dimaze obtained from the union of infinitely many disjoint non-trivial paths by identifying their first vertices and declaring their last vertices as the exits. 
An {\em outgoing comb} is a dimaze obtained from a ray, called the \emph{spine}, by adding infinitely many non-trivial disjoint paths, that meet the ray precisely at their initial vertices, and declaring the sinks of the resulting digraph to be the exits.
     
A {\em strict gammoid} is a matroid isomorphic to the one defined on the vertex set of (the digraph of) a dimaze, whose independent sets are those subsets that are linkable to the exits by a set of disjoint paths (\cite{Mas72, Per68}). Any dimaze defining a given strict gammoid is a \emph{presentation} of that strict gammoid.   
A \emph{gammoid} is a matroid restriction of a strict gammoid; so, by definition, the class of gammoids is closed under matroid deletion.

A pleasant property of the class of finite gammoids is that it is also closed under matroid contractions, and hence, under taking minors. In contrast, whether the class of all gammoids, possibly infinite, is minor-closed is an open question. 
Our aim in the first part of this paper is to begin to address this and related questions.

A standard proof of the fact that the finite gammoids are minor-closed as a class of matroids proceeds
via duality. 
The proof of this fact can be extended to infinite dimazes whose underlying (undirected) graph does not contain any ray, but it breaks down when rays are allowed. However, by developing an infinite version of a construction in \cite{AR10}, we are able to prove the following (\autoref{thm:gammoidOCMinors}).

\begin{thm0}\label{thm:OutgoingCombFree}
The class of gammoids that admit a presentation not containing any outgoing comb is minor-closed. 
\end{thm0}

If we do allow outgoing combs, combining the tools developed to prove the above theorem with a proof of Pym's linkage theorem \cite{Pym69}, we can still establish the following  (\autoref{thm:gammoidFiniteMinors}).

\begin{thm0} 
Any finite-rank minor of an infinite gammoid is a gammoid.
\end{thm0}

Outgoing combs first appeared in \cite{Car14+} where Carmesin used a topological approach to extend finite gammoids to the infinite case.
Given a dimaze $(D, B_0)$, call a set $I$ of vertices {\em topologically linkable} if there are disjoint rays or paths starting from $I$, such that the rays are the spines of outgoing combs (contained in $(D, B_0)$), and the paths end at exits or at centres of linking fans.\footnote{With a suitable topology on the undirected underlying graph, this notion of linkability has a topological interpretation in which these sets of rays and paths are precisely those that link $I$ to the topological closure of the set of the exits \cite{Car14+}. } 

		Carmesin proved that the topologically linkable sets of a dimaze form the independent sets of a finitary\footnote{
	Finitary means that a set is independent if and only if all of its finite subsets are.} matroid on the vertex set of the digraph. Any matroid isomorphic to such a matroid is called a \emph{strict topological gammoid}. A matroid restriction of a strict topological gammoid is called a \emph{topological gammoid}.   
Making use of this result together with Theorem \autoref{thm:OutgoingCombFree}, we prove the following
(\autoref{thm:characterizationFinitaryGammoids} and \autoref{thm:Top-minor-closed}).

\begin{thm0}
A matroid is a topological gammoid if and only if it is a finitary gammoid. Moreover, the class of topological gammoids is minor-closed. 
\end{thm0}

\medskip

In the second part of the paper, we turn to duality. Recall that a \emph{transversal matroid} can be defined by taking a fixed vertex class of a bipartite graph as the ground set and its matchable subsets as the independent sets. Ingleton and Piff \cite{IP73} proved constructively that finite strict gammoids and finite transversal matroids are dual to each other, a key fact to the result that the class of finite gammoids is closed under duality. 
In contrast, an infinite strict gammoid need not be dual to a transversal matroid, and vice versa (Examples \ref{thm:ARnotCotransversal} and \ref{thm:transversalNotDSG}). Despite these examples, it might still be possible that the class of infinite gammoids is closed under duality. However, we will see in \autoref{sec:discussionDuality} that there is a gammoid, which is not dual to any gammoid.

In \cite{ALM}, we introduced infinite strict gammoids that can be defined by dimazes that do not contain a type of dimaze, called an \emph{alternating comb} (see \autoref{sec:linkabilitySystem}). 
The class formed by these gammoids is rich in the sense that it contains, for example, highly connected gammoids which are neither nearly finitary nor dual to such matroids (see \cite{ALM} for the definitions and the examples). We remark that constructing a gammoid 
outside of this class is non-trivial \cite{ALM}.
There, it was also proved that,
for a dimaze, containing an alternating comb is the unique obstruction to a characterization of maximally linkable sets; in a finite dimaze, a set is maximally linkable if and only if it is linkable onto the exits.

In \autoref{sec:gammoidsTransversalDuality}, we aim to describe the dual of the strict gammoids that admit a presentation not containing any alternating comb. 
It turns out that there exists a strict gammoid in this class that is not dual to any transversal matroid. For this reason, we first extend transversal matroids to \emph{path-transversal matroids.} 
Then, using the above characterization of maximal independent sets, we prove the following (\autoref{thm:legalDuality}).

\begin{thm0}
A strict gammoid that admits a presentation not containing any alternating comb is dual to a path-transversal matroid. 
\end{thm0}

This theorem is sharp in the sense that there is a strict gammoid that is not dual to any path-transversal matroid (\autoref{thm:infinitetree-legal}). We remark that the theorem is used in \cite{ALM14} to characterize cofinitary transversal matroids and cofinitary strict gammoids. 

\section{Preliminaries}\label{sec:pre}
In this paper, digraphs do not have any loops or parallel edges.
We collect definitions, basic results and examples. For notions not found here, we refer to \cite{BDKPW} and \cite{Oxl92} for matroid theory, and \cite{Die10} for graph theory.

\subsection{Infinite matroids}
Given a set $E$ and a family of subsets $\Ical\subseteq 2^E$, let $\Ical^{\max}$ denote the maximal elements of $\Ical$ with respect to set inclusion. For a set $I\subseteq E$ and $x\in E$, we write $x, I+x, I-x$ for $\{x\}, I\cup \{x\}$ and $I\-\{x\}$ respectively. 
\begin{defi}\cite{BDKPW}
A \emph{matroid} $M$ is a pair $(E, \Ical)$, where $E$ is a set and $\Ical\subseteq 2^E$, which satisfies the following:
\begin{samepage}
\begin{enumerate}
\item[(I1)] $\emptyset\in \Ical$.
\item[(I2)] If $I\subseteq I'$ and $I'\in\Ical$, then $I\in \Ical$.
\item[(I3)] For all $I\in \Ical \setminus\Ical^{\max}$ and $I'\in \Ical^{\max}$, there is an $x\in I'\setminus I$ such that $I+x\in \Ical$.
\item[(IM)] Whenever $I\in \Ical$ and $I\subseteq X\subseteq E$, the set $\{I'\in \Ical : I\subseteq I'\subseteq X\}$ has a maximal element. 
\end{enumerate}
\end{samepage}
\end{defi}

For $M=(E, \Ical)$ a matroid, $E$ is the ground set, a subset of which is independent if it is in $\Ical$; otherwise dependent. A base is a maximal independent set, while a circuit is a minimal dependent set. Let $\Ccal(M)$ be the set of circuits of $M$. 
A circuit of size one is called a loop. 
The dual matroid $M^*$ of $M$ has as bases precisely the complements of bases of $M$. 
We usually identify a matroid with its set of independent sets, and so write an independent set $I$ is in $M$. 

Let $M=(E, \Ical)$ be a set system. The set $\Ical^{\rm fin}$ consists of the sets which have all their finite subsets in $\Ical$. $M^{\rm fin}=(E, \Ical^{\rm fin})$ is called {\em finitarisation} of $M$. $M$ is called {\em finitary} if $M = M^{\rm fin}$. Applying Zorn's Lemma one see that finitary set systems always satisfy (IM). 
\medskip

Minors are matroids obtained by deleting and contracting disjoint subsets of the ground set. We write $M.X$ for the matroid obtained by contracting the complement of a set $X\subseteq E$. 
The following standard fact simplifies investigations of minors. 

\begin{lem} \label{thm:wlogContractIndependent} 
Let $M$ be a matroid, $C, D \subseteq E$ with $C \cap D = \emptyset$ and let $M' := M / C \setminus D$ be a minor.
Then there is an independent set $S$ and a coindependent set $R$ such that $M' = M / S \setminus R$.
\end{lem}

\begin{proof}
Let $S$ be the union of a base of $M | C$ and a base of $M.D$ and let $R := (C \cup D) \setminus S$. In particular $S$ is independent (by \cite[Corollary 3.6]{BDKPW}). Since $R$ is disjoint from some base extending $S$ in $E\-(C\cup D)$, it is coindependent. In particular, any base of $M/S\- R$ spans $M/S$. 
For a set $B \subseteq E \setminus (C \cup D)$ we have: 
\begin{align*}
& B \in \Bcal(M \- D / C ) \\ 
\Leftrightarrow& B \cup (C \cap S) \in \Bcal(M \setminus D) \\
\Leftrightarrow& B \cup (C \cap S) \cup (D \cap S) \in \Bcal(M) \\
\Leftrightarrow& B \cup S \in \Bcal(M) \\
\Leftrightarrow& B \in \Bcal(M / S) \\
\Leftrightarrow& B \in \Bcal(M / S \setminus R). \qedhere
\end{align*}
\end{proof}

\subsection{Linkability system}\label{sec:linkabilitySystem}
Given a digraph $D$, let $V:=V(D)$ and $B_0\subseteq V$ be a set of sinks. Call the pair $(D,B_0)$ a {\em dimaze}\footnote{Dimaze is short for directed maze.} and $B_0$ the (set of) {\em exits}. Given a (directed) path or ray $P$, $\ini(P)$ and $\Ter(P)$ denote the initial and the terminal vertex (if exists) of $P$, respectively. For a set $\Pcal$ of paths and rays, then $\Ini(\Pcal) = \{ \ini(P) : P \in \Pcal \}$ and $\ter(\Pcal) = \{ \ter(P) : P \in \Pcal \}$. A \emph{linkage} $\Pcal$ is a set of (vertex disjoint) paths ending in $B_0$. A set $A\subseteq V$ is {\em linkable} if there is a linkage 
$\Pcal$ from $A$ to $B$, i.e.~$\Ini(\Pcal)=A$ and $\Ter(\Pcal)\subseteq B$; $\Pcal$ is \emph{onto} $B$ if $\Ter(\Pcal) = B$. 

\begin{defi}
Let $(D, B_0)$ be a dimaze. The pair of $V(D)$ and the set of linkable subsets is denoted by $M_L(D, B_0)$. 
A \emph{strict gammoid} is a matroid isomorphic to $M_L(D, B_0)$ for some $(D, B_0)$. A \emph{gammoid} is a restriction of a strict gammoid. Given a gammoid $M$, $(D, B_0)$ is called a \emph{presentation} of $M$ if $M = M_L(D, B_0)|X$ for some $X\subseteq V(D)$.
\end{defi}

In general, $M_L(D, B_0)$ satisfies (I1), (I2) and (I3) but not (IM); see \cite{ALM}. 

If $D'$ is a subdigraph of $D$ and $B_0'\subseteq B_0$, then $(D, B_0)$ contains $(D', B_0')$ as a \emph{subdimaze}. 
A dimaze $(D', B_0')$ is a \emph{subdivision} of $(D, B_0)$ if it can be obtained from $(D, B_0)$ as follows. We first add an extra vertex $b_0$ and the edges $\{ (b, b_0): b\in B_0\}$ to $D$. Then the edges of this resulting digraph are subdivided to define a digraph $D''$. Set $B_0'$ as the in-neighbourhood of $b_0$ in $D''$ and $D'$ as $D'' - b_0$. 
Note that this defaults to the usual notion of subdivision if $B_0 = \emptyset$. 

The following dimazes play an important role in our investigation. 
An undirected ray is a graph with an infinite vertex set $\{ x_i : i\geq 1\}$ and the edge set $ \{ x_i x_{i+1} : i \geq 1\}$. 
We orient the edges of an undirected ray in different ways to construct three dimazes: 
\begin{enumerate}
\item $R^A$ by orienting $(x_{i+1}, x_i)$ and $(x_{i+1}, x_{i+2})$ for each odd $i\geq 1$ and the set of exits is empty;
\item $R^I$ by orienting $(x_{i+1}, x_i)$ for each $i\geq 1$ and $x_1$ is the only exit;
\item $R^O$ by orienting $(x_i, x_{i+1})$ for each $i\geq 1$ and the set of exits is empty.
\end{enumerate}

A subdivision of $R^A$, $R^I$ and $R^O$ is called \emph{alternating ray}, \emph{incoming ray} and \emph{(outgoing) ray}, respectively. 

Let $Y = \{ y_i : i\geq 1\}$ be a set disjoint from $X$. 
We extend the above types of rays to combs by adding edges (and their terminal vertices) and declaring the resulting sinks to be the exits: 
\begin{enumerate}
\item $C^A$ by adding no edges to $R^A$;
\item $C^I$ by adding the edges $(x_i, y_i)$ to $R^I$ for each $i\geq 2$; 
\item $C^O$ by adding the edges $(x_i, y_i)$ to $R^O$ for each $i\geq 2$.
\end{enumerate}

Furthermore we define the dimaze $F^\infty$ by declaring the sinks of the digraph $(\{v, v_i : i \in \Nbb \}, \{ (v, v_i) : i \in \Nbb \})$ to be the exits. 

Any subdivision of $C^A$, $C^I$, $C^O$ and $F^\infty$ is called \emph{alternating comb}, \emph{incoming comb}, \emph{outgoing comb} and \emph{linking fan}, respectively. 
The subdivided ray in any comb is called the \emph{spine} and the paths to the exits are the \emph{spikes}. 

A dimaze $(D, B_0)$ is called \emph{$\Hcal$-free} for a set $\Hcal$ of dimazes if it does not have a subdimaze isomorphic to a subdivision of an element in $\Hcal$. A (strict) gammoid is called \emph{$\Hcal$-free} if it admits an $\Hcal$-free presentation. 

A main result in \cite{ALM} is the following.  
\begin{thm}
\label{thm:ACfree}
(i) Given a dimaze, the sets linkable onto the exits are maximally linkable if and only if the dimaze is $C^A$-free. 
(ii) Any $C^A$-free dimaze defines a strict gammoid. 
\end{thm}
In general, a $\Hcal$-free gammoid may admit a presentation that is not $\Hcal$-free (see \autoref{fig:ARIC} for $\Hcal = \{C^A\}$). 
In \cite{ALM} it was shown that the class of $C^A$-free strict gammoids contains interesting examples including highly connected matroids which are not \emph{nearly finitary} (\cite{ACF}) matroids or the dual of such ones. 
This class is also a fruitful source for \emph{wild} matroids (\cite{BC12}). 

\medskip
 Let $(D, B_0)$ be a dimaze and $\Qcal$ a set of disjoint paths or rays (usually a linkage).
A \emph{$\Qcal$-alternating walk} is a sequence $W=w_0e_0w_1e_1 \ldots$
of vertices $w_i$ and distinct edges $e_i$ of $D$ not ending with an edge, such that
every $e_i\in W$ is incident with $w_i$ and $w_{i+1}$, and the following properties hold for each $i \geq 0$ (and $i < n$ in case $W$ is finite, where $w_n$ is the last vertex):

\begin{enumerate}
\item[(W1) \phantomsection \label{aw1}] $e_i=(w_{i+1}, w_i)$ if and only if $e_i\in E(\Qcal)$;
\item[(W2) \phantomsection \label{aw2}] if $w_i=w_j$ for any $j\neq i$, then $w_i\in V(\Qcal)$;
\item[(W3) \phantomsection \label{aw3}] if $w_i\in V(\Qcal)$, then $\{e_{i-1},e_i\} \cap E(\Qcal)\neq \emptyset$ (with $e_{-1}:=e_0$).
\end{enumerate}

Let $\Pcal$ be another set of disjoint paths or rays. A \emph{$\Pcal$-$\Qcal$-alternating walk} is a $\Qcal$-alternating walk whose edges are in $E(\Pcal) \Delta E(\Qcal)$, and such that any interior vertex $w_i$ satisfies
\begin{enumerate}
\item[(W4) \phantomsection \label{aw4}] if $w_i\in V(\Pcal)$, then $\{e_{i-1},e_i\} \cap E(\Pcal)\neq \emptyset$.
\end{enumerate}

Two $\Qcal$-alternating walks $W_1$ and $W_2$ are \emph{disjoint} if they are edge disjoint, 
$V(W_1) \cap V(W_2) \subseteq V(\Qcal)$ and $\ter(W_1) \neq \ter(W_2)$.

Suppose that a dimaze $(D,B_0)$, a set $X\subseteq V$ and a linkage $\Pcal$ from a subset of $X$ to $B_0$ are given. An $X$--$B_0$ (vertex) separator $S$ is a set of vertices such that every path from $X$ to $B_0$ intersects $S$, and $S$ is {\em on} $\Pcal$ if it consists of exactly one vertex from each path in $\Pcal$.

We recall a classical result due to Gr\"unwald \cite{Gru38}, which can be formulated as follows (see also \cite[Lemmas 3.3.2 and 3.3.3]{Die10}).

\begin{lem} \label{thm:AW2Linkage}
Let $(D,B_0)$ be a dimaze, $\Qcal$ a linkage, and $\Ini(\Qcal)\subseteq X\subseteq V$. 
\begin{enumerate}
\item[(i)] If there is a $\Qcal$-alternating walk from $X\- \Ini(\Qcal)$ to $B_0 \- \Ter(\Qcal)$, then there is a linkage $\Qcal'$ with $\Ini(\Qcal)
\subsetneq \Ini(\Qcal')\subseteq X$ onto $\Ter(\Qcal)\subsetneq \Ter(\Qcal')\subseteq B_0$.

\item[(ii)] If there is not any $\Qcal$-alternating walk from $X\- \Ini(\Qcal)$ to $B_0 \- \Ter(\Qcal)$, then there is a $X$--$B_0$ separator on $\Qcal$.
\end{enumerate}
\end{lem}

A set $X\subseteq V$ in $(D, B_0)$ is \emph{topologically linkable} if $X$ admits a \emph{topological linkage}, which means that from each vertex $x\in X$, there is a \emph{topological path} $P_x$, i.e.~$P_x$ is the spine of an outgoing comb, a path ending in the centre of a linking fan, or a path ending in $B_0$, such that $P_x$ is disjoint from $P_y$ for any $y\neq x$. Clearly, a finite topologically linkable set is linkable. 
Denote by $M_{TL}(D, B_0)$ the pair of $V$ and the set of the topologically linkable subsets. 
Carmesin gave the following connection between dimazes (not necessarily defining a matroid) and topological linkages.

\begin{cor}\cite[Corollary 5.7]{Car14+}\label{thm:finitarizationOfGammoid}
Given a dimaze $(D, B_0)$, $M_{TL}(D, B_0) = M_L(D, B_0)^{\rm fin}$. In particular, $M_{TL}(D, B_0)$ is always a finitary matroid.
\end{cor}

A \emph{strict topological gammoid} is a matroid of the form $M_{TL}(D, B_0)$, and a restriction of which is called a \emph{topological gammoid}.

\subsection{Transversal system}
Let $G= (V, W)$ be a bipartite graph and call $V$ and $W$, respectively, the {\em left} and the {\em right} vertex class of $G$. 
A subset $I$ of $V$ is {\em matchable onto $W'\subseteq W$} if there is a matching $m$ of $I$ such that $m\cap V=I$ and $m\cap W=W'$; where we are identifying a set of edges (and sometimes more generally a subgraph) with its vertex set. Given a set $X\subseteq V$ or $X\subseteq W$, we write $m(X)$ for the set of vertices matched to $X$ by $m$ and $m\upharpoonright X$ for the subset of $m$ incident with vertices in $X$.

\begin{defi}
Given a bipartite graph $G=(V,W)$,  
the pair of $V$ and all its matchable subsets is denoted by $M_T(G)$. 
A \emph{transversal matroid} is a matroid isomorphic to $M_T(G)$ for some $G$. 
Given a transversal matroid $M$, $G$ is a \emph{presentation} of $M$ if $M = M_T(G)$.
\end{defi}

In general, a transversal matroid may have different presentations. The following is a well-known fact (see \cite{BS68}). 

\begin{lem}
\label{thm:coversRHS}
Let $G=(V, W)$ be a bipartite graph. Suppose there is a maximal element in $M_T(G)$, witnessed by a matching $m_0$. Then $M_T(G)=M_T(G\- (W-m_0))$, and $N(W - m_0)$ is a subset of every maximal element in $M_T(G)$. 
\end{lem}

In case $M_T(G)$ is a matroid, the second part states that $N(W-m_0)$ is a set of coloops. From now on, wherever there is a maximal element in $M_T(G)$, we assume that $W$ is covered by a matching. 

If $G$ is finite, Edmonds and Fulkerson \cite{EF65} showed that $M_T(G)$ satisfies (I3), and so is a matroid. When $G$ is infinite, $M_T(G)$ also satisfies (I3) (\cite{ALM}) but need not be a matroid. 
Given a matching $m$, an \emph{$m$-alternating walk} is a walk such that the consecutive edges alternate in and out of $m$ in $G$. Given another matching $m'$, an $m$-$m'$-{\it alternating walk} is a walk such that consecutive edges alternate between the two matchings. 

A standard compactness proof shows that a left locally finite bipartite graph $G=(V,W)$, i.e.~every vertex in $V$ has finite degree, defines a finitary transversal matroid. 
\begin{lem}[\cite{MP67}]
\label{thm:finitaryTransversal}
Every left locally finite bipartite graph defines a finitary transversal matroid.
\end{lem}

The following corollary is a tool to show that a matroid is not transversal.
 
\begin{lem}\label{thm:TransversalInfiniteCircuit}
Any infinite circuit of a transversal matroid contains an element which does not lie in any finite circuit. 
\end{lem}
\begin{proof}
Let $C$ be an infinite circuit of some $M_T(G)$. Applying \autoref{thm:finitaryTransversal} on the restriction of $M_T(G)$ to $C$, we see that there is a vertex in $C$ having infinite degree. However, such a vertex does not lie in any finite circuit. 
\end{proof}

\section{Minor}\label{sec:minor}
The class of gammoids is closed under deletion by definition. In fact, finite gammoids are minor-closed. To see this, note that matroid deletion and contraction commute, so it suffices to show that a contraction minor $M/X$ of a strict gammoid $M$ is also a gammoid. Indeed, in \cite{IP73} it was shown that finite strict gammoids are precisely the dual of finite transversal matroid. Moreover, they provided a construction to turn a dimaze to a bimaze presentation of the dual, and vice versa (essentially Definitions \ref{def:dualityConstruction} and \ref{def:conDimaze}). Thus, we apply the construction to a presentation of $M$ and get one of $M^*$. 
By deleting $X$, we get a presentation of the transversal matroid $M^*\- X$. Reversing the construction with any base of $M^*\- X$ gives us  a dimaze presentation of $(M^*\setminus X)^* = M/X$. 

In case of general gammoids, we can no longer appeal to duality, since, as we shall see, strict gammoids need not be cotransversal (\autoref{thm:ARnotCotransversal}) and the dual of transversal matroids need not be strict gammoids (\autoref{thm:transversalNotDSG}). 
We will instead investigate the effect of the construction sketched above on a dimaze directly.  
We are then able to show that the class of $C^O$-free gammoids, i.e.~gammoids that admit a $C^O$-free presentation, is minor-closed. 
In combination with the linkage theorem, we can also prove that finite rank minors of gammoids are gammoids.
It remains open whether the class of gammoids is closed under taking minors.

\medskip
Topological gammoids are introduced in \cite{Car14+} in response to a question raised by Diestel. The independent set systems are always finitary and define matroids. It turns out that such matroids are precisely the finitary gammoids. By investigating the structure of dimaze presentations of such gammoids, we then show that finitary strict gammoids, or equivalently, topological gammoids, are also closed under taking minors. 

\subsection{Matroid contraction and shifting along a linkage}
Our aim is to show that a contraction minor $M/S$ of a strict gammoid $M$ is a strict gammoid. By \autoref{thm:wlogContractIndependent}, we may assume that $S$ is independent. 
The first case is that $S$ is a subset of the exits. 
\begin{lem}\label{thm:contractSubsetOfB_0}
Let $M = M_L(D, B_0)$ 
be a strict gammoid and $S \subseteq B_0$. Then a dimaze presentation of $M / S$ is given by $M_L(D - S, B_0 \setminus S)$.
\end{lem}
\begin{proof}
Since $S\subseteq B_0$ is independent, $I\in \Ical(M/S) \iff I \cup S \in \Ical(M)$. Moreover, 
\begin{align*}
I\in \Ical(M/S)	&\iff I \cup S \mbox{ admits a linkage in $(D, B_0)$ } \\
					&\iff I \mbox{ admits a linkage $\Qcal$ with $\Ter(\Qcal)\cap S = \emptyset$ in $(D, B_0)$}\\
					&\iff I \in \Ical( M_L(D - S, B_0\- S).  \qedhere
\end{align*}
\end{proof}

Thus, it suffices to give a dimaze presentation of $M$ such that $S$ is a subset of the exits. For this purpose we consider the process of ``shifting along a linkage'', which replaces the previously discussed detour via duality.

\medskip 
Throughout the section, $(D, B_0)$ denotes a dimaze, $\Qcal$ a set of disjoint paths or rays, $S := \ini(\Qcal)$ and $T := \ter(\Qcal)$.  
Next, we define various maps which are dependent on $\Qcal$.  

Define a bijection between $V\- T$ and $V\- S$ as follows: $\vec{\Qcal}(v) : = v $ if $v\notin V(\Qcal)$; otherwise $\vec{\Qcal}(v) := u$ where $u$ is the unique vertex such that $(v, u)\in E(\Qcal)$. The inverse is denoted by $\cev{\Qcal}$. 

Construct the digraph $\vec{\Qcal}(D)$ from $D$ by replacing each edge $(v, u) \in E(D) \setminus E(\Qcal)$ with $(\vec{\Qcal}(v), u)$ and each edge $(v, u) \in \Qcal$ with $(u, v)$.
Set for the rest of this section 
\[D_1 := \vec{\Qcal}(D) \mbox{ and } B_1 := (B_0 \setminus T) \cup S\]
and call $(D_1, B_1)$ the \emph{$\Qcal$-shifted} dimaze.

Given a $\Qcal$-alternating walk $W = w_0 e_0 w_1 e_1 w_2 \ldots$ in $D$, let $\vec{\Qcal}(W)$ be obtained from $W$ by deleting all $e_i$ and each $w_i \in W$ such that $w_i \in V(\Qcal)$ but $e_i \notin E(\Qcal)$. 

For a path or ray $P = v_0 v_1 v_2 \ldots$ in $D_1$, let $\cev{\Qcal}(P)$ be  obtained from $P$ by inserting after each $v_i \in P \setminus \ter(P)$  the following:
\begin{description}
\item[$(v_i, v_{i+1})$] if $v_i \notin V(\Qcal)$;
\item[$(v_{i+1}, v_i)$] if $v_i \in V(\Qcal)$ and $(v_{i+1}, v_i) \in E(\Qcal)$;
\item[$(w, v_i) w (w, v_{i+1})$] with $w := \cev{\Qcal}(v_i)$ if $v_i \in V(\Qcal)$ but $(v_{i+1}, v_i) \notin E(\Qcal)$.
\end{description}

\begin{figure}
\begin{center}
\includegraphics[scale=0.9]{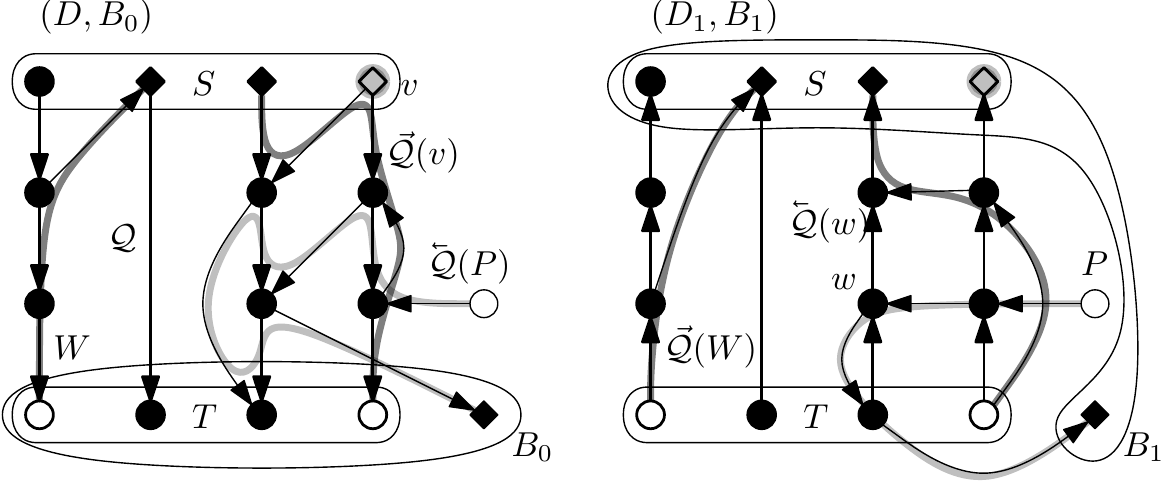}
\end{center}
\caption{A $\Qcal$-shifted dimaze: $D_1=\vec{\Qcal}(D)$, $B_1 = (B_0 \setminus T) \cup S$, where $\Qcal$ consists of the vertical downward paths. 
Outlined circles and diamonds are respectively initial and terminal vertices of $\Qcal$-alternating walks (left) and their $\vec{\Qcal}$-images (right). }
\label{fig:Qcal_shift}
\end{figure}

We examine the relation between alternating walks in $D$ and paths/rays in~$\vec{\Qcal}(D)$. 

\begin{lem}\label{thm:AW-path_shift}
\begin{enumerate}
\item[(i)] A $\Qcal$-alternating walk in $D$ that is infinite or ends in $t\in B_1$ is respectively mapped by $\vec{\Qcal}$ to a ray or a path ending in $t$ in $D_1$. Disjoint such walks are mapped to disjoint paths/rays. 
\item[(ii)]  A ray or a path ending in $t \in B_1$ in $D_1$ is respectively mapped by $\cev{\Qcal}$ to an infinite $\Qcal$-alternating walk or a finite $\Qcal$-alternating walk ending in $t$ in $D$. Disjoint such paths/rays are mapped to disjoint $\Qcal$-alternating walks.
\end{enumerate}
\end{lem}

\begin{proof}
We prove (i) since a proof of (ii) can be obtained by reversing the construction. 

Let $W = w_0 e_0 w_1 e_1 w_2 \ldots$ be a $\Qcal$-alternating walk in $D$.  
If a vertex $v$ in $W$ is repeated, then $v$ occurs twice and there is $i$ such that $v = w_i$ with $e_{i-1} = (w_i, w_{i-1}) \in E(\Qcal)$ and $e_i \notin E(\Qcal )$. Hence, $w_i$ is deleted in $P:=\vec{\Qcal}(W)$ and so $v$ does not occur more than once in $P$, that is, $P$ consists of distinct vertices.

By construction, the last vertex of a finite $W$ is not deleted, hence $P$ ends in $t$. In case $W$ is infinite, by \hyperref[aw3]{(W3)}, no tail of $W$ is deleted so that $P$ remains infinite.

Next, we show that $(v_i, v_{i+1})$ is an edge in $D_1$. Let $w_j = v_i$ be the non-deleted instance of $v_i$. 
If $w_{j+1}$ has been deleted, then the edge $(w_{j+1}, w_{j+2})$ (which exists since the last vertex cannot be deleted) in $D$ has been replaced by the edge $(\vec{\Qcal}(w_{j+1}), w_{j+2}) = (v_i, v_{i+1})$ in $D_1$.
If both $w_j$ and $v_{i+1} = w_{j+1}$ are in $V(\Qcal)$ then the edge $(w_{j+1}, w_j) \in E(\Qcal)$ has been replaced by $(v_i, v_{i+1})$ in $D_1$.
In the other cases $(w_j, w_{j+1}) = (v_i, v_{i+1})$ is an edge of $D$ and remains one in $D_1$.

Let $W_1, W_2$ be disjoint $\Qcal$-alternating walks. 
By construction, $\vec{\Qcal}(W_1) \cap \vec{\Qcal}(W_2) \subseteq W_1 \cap W_2 \subseteq V(\Qcal)$. 
By disjointness, at any intersecting vertex, one of $W_1$ and $W_2$ leaves with an edge not in $E(\Qcal)$. Thus, such a vertex is deleted upon application of $\vec{\Qcal}$. Hence, $\vec{\Qcal}(W_1)$ and $\vec{\Qcal}(W_2)$ are disjoint paths/rays. 
\end{proof}

Note that for a path $P$ in $D_1$ and a $\Qcal$-alternating walk $W$ in $D$, we have
\begin{eqnarray*}
\vec{\Qcal}(\cev{\Qcal}(P)) = P;&
&\cev{\Qcal}(\vec{\Qcal}(W)) = W.
\end{eqnarray*} 

This correspondence of sets of disjoint $\Qcal$-alternating walks in $(D, B_0)$ and sets of disjoint paths or rays in the $\Qcal$-shifted dimaze will be used in various situations in order to show that the independent sets associated with $(D, B_0)$ and the $\Qcal$-shifted dimaze are the same.

Given a set $\Wcal$ of $\Qcal$-alternating walks, define the graph $\Qcal \Delta \Wcal := (V(\Qcal)\cup V(\Wcal), E(\Qcal) \Delta E(\Wcal))$.

\begin{lem}\label{thm:aws2paths}
Let $J\subseteq V \setminus S$ and $\Wcal$ a set of disjoint $\Qcal$-alternating walks, each of which starts from $J$ and does not end outside of $B_1$. Then there is a set of disjoint rays or paths from $X := J \cup (S \setminus \ter(\Wcal))$ to $Y := T \cup (\ter(\Wcal) \cap B_0)$ in $\Qcal \Delta \Wcal$.
\end{lem}

\begin{proof}
Every vertex in $\Qcal \Delta \Wcal  \setminus (X \cup Y)$ has in-degree and out-degree both $1$ or both $0$. Moreover, every vertex in $X$ has in-degree $0$ and out-degree $1$ (or $0$, if it is also in $Y$) and every vertex in $Y$ has out-degree $0$ and in-degree $1$ (or $0$, if it is also in $X$).
Therefore every (weakly) connected component of $\Qcal \Delta \Wcal$ meeting $X$ is either a path ending in $Y$ or a ray. 
\end{proof}

The following will be used to complete a ray to an outgoing comb in various situations.

\begin{lem} \label{thm:OR2OC}
Suppose $\Qcal$ is a topological linkage. Any ray $R$ that hits infinitely many vertices of $V(\Qcal)$ is the spine of an outgoing comb.
\end{lem}

\begin{proof}
The first step is to inductively construct an infinite linkable subset of $V(R)$.
Let $\Qcal_0 :=\Qcal$ and $A_0 := \emptyset$. 
For $i \geq 0$, assume that $\Qcal_i$ is a topological linkage that intersects $V(R)$ infinitely but avoids the finite set of vertices $A_i$. 
Since it is not possible to separate a vertex on a topological path from $B_0$ by a finite set of vertices disjoint from that topological path, there exists a path $P_i$ from $V(R) \cap V(\Qcal_i)$ to $B_0$ avoiding $A_i$. 
Let $A_{i+1} := A_i \cup V(P_i)$ and $\Qcal_{i+1}$ obtained from $\Qcal_i$ by deleting from each of its elements the minimal initial segment that intersects $A_{i+1}$. 
As $\Qcal_{i+1}$ remains a topological linkage that intersects $V(R)$ infinitely, we can continue the procedure.
By construction $\{ P_i: i \in \Nbb\}$ is an infinite set of disjoint finite paths from a subset of $V(R)$ to $B_0$. Let $p_i \in P_i$ be the last vertex of $R$ on $P_i$, then $R$ is the spine of the outgoing comb:
$\displaystyle{R \cup \bigcup_{i \in \Nbb} p_i P_i}$.
\end{proof}

\begin{cor}\label{thm:ORinSymmetricDifference2OC}
Any ray provided by \autoref{thm:aws2paths} is in fact the spine of an outgoing comb if $\Qcal$ is a topological linkage, and the infinite forward segments of the walks in $\Wcal$ are the spines of outgoing combs.
\end{cor}

\begin{proof}
Observe that a ray $R$ constructed in \autoref{thm:aws2paths} is obtained by alternately following the forward segments of the walks in $\Wcal$ and the forward segments of elements in $\Qcal$. 

Either a tail of $R$ coincides with a tail of a walk in $\Wcal$, and we are done by assumption; or $R$ hits infinitely many vertices of $V(\Qcal)$, and \autoref{thm:OR2OC} applies.
\end{proof}

With \autoref{thm:aws2paths} we can transform disjoint alternating walks into disjoint paths or rays. A reverse transform is described as follows. 

\begin{lem}\label{thm:maximalPQAWsDisjoint}
Let $\Pcal$ and $\Qcal$ be two sets of disjoint paths or rays.
 Let $\Wcal$ be a set of maximal $\Pcal$-$\Qcal$-alternating walks starting in distinct vertices of $\ini(\Pcal)$. Then the walks in $\Wcal$ are disjoint and 
can only end in $(\ter(\Pcal) \setminus T) \cup S$. 
\end{lem}

\begin{proof}
Let $W = w_0 e_0 w_1 \ldots$ be a maximal $\Pcal$-$\Qcal$-alternating walk. Then $W$ is a trivial walk if and only if $w_0 \in (\ter(\Pcal) \setminus T) \cup S$. If $W$ is nontrivial then $e_0 \in E(\Qcal)$ if and only if $w_0 \in V(\Qcal)$.

Let $W_1$ and $W_2\in \Wcal$. 
Note that for any interior vertex $w_i$ of a $\Pcal$-$\Qcal$-alternating walk, it follows from the definition that either edge in $\{e_{i-1}, e_i \}$ determines uniquely the other. 
So if $W_1$ and $W_2$ share an edge, then a reduction to their common initial vertex shows that they are equal by their maximality. 
Moreover if the two walks share a vertex $v\notin V(\Qcal)$, then they are equal since they share the edge of $\Pcal$ whose terminal vertex is~$v$.

Therefore, if $W_1\neq W_2$ and they end at the same vertex $v$, then $v\in V(\Pcal) \cap V(\Qcal)$. More precisely, we may assume that $v$ is the initial vertex of an edge in $E(\Qcal) \cap E(W_1)$ and the terminal vertex of an edge $e \in E(\Pcal) \cap E(W_2)$ (both the last edges of their alternating walk). Since $v$ is the initial vertex of some edge, it cannot be in $B_0$, so the path (or ray) in $\Pcal$ containing $e$ does not end at $v$. Hence we can extend $W_1$ contradicting its maximality.

Similarly we can extend a $\Pcal$-$\Qcal$-alternating walk that ends in some vertex $v \in \ter(\Pcal) \cap \ter(\Qcal)$ by the edge in $E(\Qcal)$ that has $v$ as its terminal vertex, unless $v\in \ini(\Qcal)$.
So $\Wcal$ is a set of disjoint $\Pcal$-$\Qcal$-alternating walks that can only end in $(\ter(\Pcal) \setminus T) \cup S$.
\end{proof}

Now we investigate when a dimaze and its $\Qcal$-shifted dimaze present the same strict gammoid.

\begin{lem}\label{thm:linkableInShifted2linkableInDB0}
Suppose that $\Qcal$ is a linkage from $S$ onto $T$ and $I$ a set linkable in $(D_1, B_1)$. Then $I$ is linkable in $(D, B_0)$ if (i) $I \setminus S$ is finite or (ii) $(D, B_0)$ is $C^O$-free.
\end{lem}

\begin{proof}
There is a set of disjoint finite paths from $I$ to $B_1$ in $(D_1, B_1)$, which, by \autoref{thm:AW-path_shift}, gives rise to a set of disjoint finite $\Qcal$-alternating walks from $I$ to $B_1$ in $(D, B_0)$.
Let $\Wcal$ be the subset of those walks starting in $J := I \setminus S$. Then \autoref{thm:aws2paths} provides a set $\Pcal$ of disjoint paths or rays from $J \cup (S \setminus \ter(\Wcal)) \supseteq I$ to $Y \subseteq B_0$. 
It remains to argue that $\Pcal$ does not contain any ray. Indeed, any such ray meets infinitely many paths in $\Qcal$. But by \autoref{thm:OR2OC}, the ray is the spine of an outgoing comb, which is a contradiction. 
\end{proof}

In fact the converse of (ii) holds.

\begin{lem}\label{thm:linkableInDB02linkableInShifted_noOC}
Suppose that $(D, B_0)$ is $C^O$-free, and $\Qcal$ is a linkage from $S$ onto $T$ such that there exists no linkage from $S$ to a proper subset of $T$. Then a linkable set $I$ in $(D, B_0)$ is also linkable in $(D_1, B_1)$,
and $(D_1, B_1)$ is $C^O$-free.
\end{lem}

\begin{proof}
For the linkability of $I$ it suffices by \autoref{thm:AW-path_shift} to construct a set of disjoint finite $\Qcal$-alternating walks from $I$ to $B_1$.
Let $\Pcal$ be a linkage of $I$ in $(D, B_0)$. 

For each vertex $v \in I$ let $W_v$ be the maximal $\Pcal$-$\Qcal$-alternating walk starting in $v$. By \autoref{thm:maximalPQAWsDisjoint}, $\Wcal := \{ W_v: v \in I \}$ is a set of disjoint $\Qcal$-alternating walks that can only end in $(\ter(\Pcal) \setminus T) \cup S \subseteq B_1$. 

If there is an infinite alternating walk $W = W_{v_0}$ in $\Wcal$, then \autoref{thm:aws2paths} applied on just this walk gives us a set $\Rcal$ of disjoint paths or rays from 
$S + v_0$ to $T$. 
Since the forward segments of $W$ are subsegments of paths in $\Pcal$, by \autoref{thm:ORinSymmetricDifference2OC} any ray in $\Rcal$ would extend to a forbidden outgoing comb.
Thus, $\Rcal$ is a linkage of $S+v_0$ to $T$. In particular, $S$ is linked to a proper subset of $T$ contradicting the minimality of $T$. 
Hence $\Wcal$ consists of finite disjoint $\Qcal$-alternating walks, as desired.

For the second statement suppose that $(D_1, B_1)$ contains an outgoing comb whose spine $R$ starts at $v_0 \notin S$. Then $W :=\cev{\Qcal}(R)$ is a $\Qcal$-alternating walk in $(D, B_0)$ by \autoref{thm:AW-path_shift}.
Any infinite forward segment $R'$ of $W$ contains an infinite subset linkable to $B_1$ in $(D_1, B_1)$. By \autoref{thm:linkableInShifted2linkableInDB0}(ii) this subset is also linkable in $(D, B_0)$, so $R'$ is the spine of an outgoing comb by \autoref{thm:OR2OC}, which is a contradiction. 

On the other hand, suppose that $W$ does not have an infinite forward tail. By investigating $W$ as we did with $W_{v_0}$ above, we arrive at a contradiction. Hence, there does not exist any outgoing comb in $(D_1, B_1)$.
\end{proof}

For later applications, we note the following refinement. 

\begin{cor}\label{thm:noFanInShifted}
If $(D, B_0)$ is $F^\infty$-free as well, then so is $(D_1, B_1)$.
\end{cor}
\begin{proof}
Suppose that $(D_1, B_1)$ contains a subdivision of $F^\infty$ with centre $v_0$. Then an infinite subset $X$ of the out-neighbourhood of $v_0$ in $(D_1, B_1)$ is linkable. By \autoref{thm:linkableInShifted2linkableInDB0}(ii), $X$ is also linkable in $(D, B_0)$. As $X$ is a subset of the out-neighbourhood of $\cev{\Qcal}(v_0)$, a forbidden linking fan in $(D, B_0)$ results.
\end{proof}

\begin{pro}\label{thm:S_shiftedToExits}
Suppose $(D, B_0)$ is $C^O$-free and 
$\Qcal$ is a linkage from $S$ onto $T$ such that $S$ cannot be linked to a proper subset of $T$. 
Then $M_L(D_1, B_1) = M_L(D, B_0)$. 
\end{pro}
\begin{proof}
By \autoref{thm:linkableInShifted2linkableInDB0}(ii) and \autoref{thm:linkableInDB02linkableInShifted_noOC}, a set $I \subseteq V$ is linkable in $(D, B_0)$ if and only if it is linkable in $(D_1, B_1)$.
\end{proof}

We remark that in order to show that $M_L(D, B_0) = M_L(D_1, B_1)$, the assumption in \autoref{thm:S_shiftedToExits} that $(D, B_0)$ is $C^O$-free can be slightly relaxed. 
Only outgoing combs constructed in the proofs of \autoref{thm:linkableInShifted2linkableInDB0}(ii) and \autoref{thm:linkableInDB02linkableInShifted_noOC} which have the form that all the spikes are terminal segments of paths in the linkage $\Qcal$ need to be forbidden.

\begin{thm}\label{thm:gammoidOCMinors}
The class of $C^O$-free gammoids is minor-closed.
\end{thm}

\begin{proof}
Let $N := M_L(D, B_0)$ be a strict gammoid. 
It suffices to show that any minor $M$ of $N$ is a gammoid. 
By \autoref{thm:wlogContractIndependent}, we have $M= N/S \- R$ for some independent set $S$ and coindependent set $R$. 
First extend $S$ in $B_0$ to a base $B_1$.
This gives us a linkage $\Qcal$ from $S$ onto $T := B_0 \setminus B_1$ such that there exists no linkage from $S$ to a proper subset of $T$. 

Assume that $(D, B_0)$ is $C^O$-free. Then by \autoref{thm:linkableInDB02linkableInShifted_noOC}, $(D_1, B_1)$ is $C^O$-free,  and by \autoref{thm:S_shiftedToExits}, $M_L(D, B_0) = M_L(D_1, B_1)$.  Since $S\subseteq B_1$, $M = M_L(D_1, B_1) / S \setminus R = M_L(D_1 - S, B_1 \setminus S)\setminus R$ is a $C^O$-free gammoid.
\end{proof}

A partial converse of \autoref{thm:linkableInShifted2linkableInDB0}(i) can be proved by analyzing the proof of the linkage theorem. 

\begin{lem}\label{thm:linkableInMinor2linkableInShifted_finite}
Let $M = M_L(D, B_0)$ be a strict gammoid, $\Qcal$ a linkage from $S$ onto $T$ such that $B_1$ is a base and $I \subseteq V \setminus S$ such that $S \cup I$ is linkable in $(D, B_0)$. If $I$ is finite, then it is linkable in $(D_1 - S, B_1 \setminus S)$.
\end{lem}

\begin{proof}
By \autoref{thm:AW-path_shift} it suffices to construct a set of disjoint finite $\Qcal$-alternating walks from $I$ to $B_0 \setminus T$.

Let $\Pcal$ be a linkage of $S \cup I$ in $(D, B_0)$. We apply the linkage theorem of Pym \cite{Pym69} to get a linkage $\Qcal^\infty$ from $S \cup I$ onto some set $Y^\infty \supseteq T$ in the following way (the notation used here was introduced in \cite{ALM}):

For each $x\in S \cup I$, let $P_x$ be the path in $\Pcal$ containing $x$ and $f^0_x := x$. Let $\Qcal^0 := \Qcal$.
For each $i> 0$ and each $x\in S \cup I$, let $f_x^i$ be the last vertex $v$ on $f_x^{i-1}P_x$ such that $(f_x^{i-1} P_x\mathring{v}) \cap V(\Qcal^{i-1}) = \emptyset$.
For $y\in T$, let $Q_y$ be the path in $\Qcal$ containing $y$ and $t_y^i$ be the first vertex $v\in Q_y$ such that the terminal segment $\mathring{v}Q_y$ does not contain any $f_x^i$. Define the linkage $\Qcal^i := \Bcal^i \cup \Ccal^i$ with

\vspace{-0.2in}
\begin{align*}
\Bcal^{i} &:= \{ P_xf_x^iQ_y: x\in S \cup I, y\in T \mbox{ and } f_x^i =t_y^i\} \text{, }\\
\Ccal^{i} &:= \{ P_x\in \Pcal: f_x^i\in B_0 \setminus T\}.
\end{align*}
There exist integers $i_x, i_y\geq 0$ such that $f_x^{i_x} = f_x^k$, $t_y^{i_y} = t_y^l$ for all integers $k\geq i_x$ and $l\geq i_y$. Define $f^\infty_x := f^{i_x}_x, t^\infty_y := t^{i_y}_y$ and 

\vspace{-0.2in}
\begin{align*}
\Bcal^{\infty} &:= \{ P_xf_x^\infty Q_y: x\in S \cup I, y\in T \mbox{ and } f_x^\infty =t_y^\infty\},\\
\Ccal^{\infty} &:= \{ P_x\in \Pcal: f_x^\infty\in B_0 \setminus T\}.
\end{align*}
Then $\Qcal^\infty:=
\Bcal^\infty \cup \Ccal^\infty$ is the linkage given by the linkage theorem. 

Let $Y := Y^\infty \setminus T$ and $B_2$ an extension to a base of the independent set $(B_0 \setminus Y^\infty) \cup (S \cup I)$ inside $B_1$. Then $B_2 \setminus B_1 = I$ and $B_1 \setminus B_2 \subseteq Y$ and so, by \cite[Lemma 3.7]{BDKPW}, $\left|I\right| = \left|B_2 \setminus B_1\right| = \left|B_1 \setminus B_2\right| \leq \left|Y\right|$.

\bigskip
 Let $v \in V(D)$ be a vertex with the property that $v = f_{x_{j+1}}^{j+1}$ for some integer $j$ and a vertex $x_{j+1} \in S \cup I$ such that $f_{x_{j+1}}^j \neq f_{x_{j+1}}^{j+1}$. We backward inductively construct a walk $W(v)$ that starts from $I$ and ends in $v$ as follows:

Given $x_{i+1}$ for a positive integer $i \leq j$, let $Q_i$ be the path in $\Qcal$ containing $f_{x_{i+1}}^{i}$ (if there is no such path, then $f_{x_{i+1}}^i \in I$ and $i = 0$). 
Since $f_{x_{i+1}}^i \neq f_{x_{i+1}}^{i+1}$, it follows that $\Fcal^i\cap \mathring f_{x_{i+1}}^{i} Q_i\neq \emptyset$, where $\Fcal^i := \{ f_x^i : x \in S \cup I\}$. 
Let $x_i$ be such that $f_{x_i}^i$ is the first vertex of $\Fcal^i$ on $\mathring f_{x_{i+1}}^{i} Q_i$. 
Moreover, since $f_{x_{i+1}}^i \in Q_i$, $\Fcal^{i-1} \cap \mathring f_{x_{i+1}}^{i} Q_i = \emptyset$, so $f_{x_i}^{i-1} \neq f_{x_i}^i$. Hence we can complete the construction down to $i=1$ and define:
\begin{align}
\label{eqn:unhappinessCoin}
 W(v) := f_{x_1}^0 P_1 f_{x_1}^1 \cup \bigcup_{0 < i < j} f_{x_{i+1}}^i Q_i f_{x_i}^{i}   \cup  f_{x_{i+1}}^i P_{i+1} f_{x_{i+1}}^{i+1} . 
 \end{align}

Note that $f_{x_1}^0 \neq f_{x_1}^1$ and for any $x\in S$, the definition of $f_x^1$ implies $f_x^0 = f_x^1$. Hence, $f_{x_1}^0$, the initial vertex of $W(v)$, is in $(S \cup I) \setminus S = I$.
Now we examine the interaction between two such walks:

\begin{clmnn}
Let $x, x' \in S \cup I$ be given such that $f_x^{j+1} \neq f_x^j$ and $f_{x'}^{j'+1} \neq f_{x'}^{j'}$.

\phantomsection\label{clm:distinctInitial}
(i) If $j = j'$ and $f_x^{j+1} \neq f_{x'}^{j'+1}$, then $\ini(W(f_x^{j+1})) \neq \ini(W(f_{x'}^{j'+1}))$.

\phantomsection\label{clm:sameInitial=subwalk}
(ii) If $W(f_x^{j+1})$ and $W(f_{x'}^{j'+1})$ start at the same vertex in $I$, then one is a subwalk of the other.
\end{clmnn}

\begin{proof}
For \hyperref[clm:distinctInitial]{(i)} we first note that $f_x^{j+1}$ and $f_{x'}^{j'+1}$ are on distinct paths in $\Pcal$ and apply induction on $j$. If $j = j' = 0$, then $\ini(W(f_x^{j+1})) = x \neq x' = \ini(W(f_{x'}^{j'+1}))$. For $j>0$ the walk $W(f_{x'}^{j'+1})$ has the form $W(f_{x'_j}^{j}) \cup f_{x'_{j+1}}^j Q'_j f_{x'_j}^{j}   \cup  f_{x'_{j+1}}^j P'_{j+1} f_{x'_{j+1}}^{j+1}$ and analogue $W(f_x^{j+1})$. The vertices $f_{x'_{j+1}}^j$ and $f_{x_{j+1}}^j$ are on distinct paths in $\Pcal$ and therefore distinct. Then it follows from the definition that $f_{x_j}^{j} \neq f_{x'_j}^{j}$ and we use the induction hypothesis to see that $\ini(W(f_{x'_j}^{j})) \neq \ini(W(f_{x_j}^{j}))$ and hence $\ini(W(f_{x'_{j+1}}^{j+1})) \neq \ini(W(f_{x_{j+1}}^{j+1}))$, as desired.

For \hyperref[clm:sameInitial=subwalk]{(ii)} suppose that $f_x^{j+1} \neq f_{x'}^{j'+1}$, then \hyperref[clm:distinctInitial]{(i)} implies $j \neq j'$, say $j < j'$. If $f_{x'_{j+1}}^{j+1} \neq f_{x_{j+1}}^{j+1}$, then ,by \hyperref[clm:distinctInitial]{(i)}, $\ini(W(f_x^{j+1})) \neq \ini(W(f_{x'_{j+1}}^{j+1})) = \ini(W(f_{x'}^{j'+1}))$. Hence $W(f_x^{j+1})$ is a subwalk of $W(f_{x'}^{j'+1})$.
\end{proof}

Each vertex $y \in Y \setminus I$ is on a non-trivial path in $\Qcal^\infty$, so there exists a least integer $i_y>0$ such that $y = f_{x_{i_y}}^{i_y}$ for some $x_{i_y} \in S \cup I$. For $y \in Y \cap I$ let $W(y)$ be the trivial walk at $y$, so that we can define $\Wcal := \{ W(y): y\in Y\}$.

Suppose $y$ and $y'$  are distinct vertices in $Y \setminus I$ such that $\ini(W(y)) = \ini(W(y'))$. Since there is no edge of $\Qcal$ ending in either of these vertices, \hyperref[clm:sameInitial=subwalk]{(ii)} implies that $W(y) = W(y')$ and therefore $y = y'$. Since the initial vertex of a non-trivial walk in $\Wcal$ is not in $B_0$, we have $\ini(W(y)) \neq \ini(W(y'))$ for any two distinct vertices $y, y'$ in $Y$. That means  $\ini(\Wcal) = I$, since $\left|I\right| \leq \left|Y\right|$.  

By \autoref{thm:maximalPQAWsDisjoint}, the maximal $\Qcal^\infty$-$\Qcal$-alternating walks starting in $I$ are disjoint. Thus, to complete the proof, it remains to check that each $\Qcal^\infty$-$\Qcal$-alternating walk starting in $I$ is finite. To that end, let $e$ be an edge of such a walk. As $E(\Wcal)$ is finite, it suffices to show that $e\in E(W)$ for some $W\in \Wcal$. 
By definition, $e \in E(\Qcal^\infty) \Delta E(\Qcal)$. 
The following case analysis completes the proof. 
\begin{enumerate}
\item $e \in E(\Qcal^\infty) \setminus E(\Qcal)$: $e$ is on some initial segment $P_x f_x^\infty$ of a path $P_x$ in $\Pcal$. More precisely, there is an integer $i$, such that $e \in f_x^i P_x f_x^{i+1}$. By construction $e \in W(f_x^{i+1})$ and $\ini(W(f_x^{i+1})) \in I$. Let $W$ be the walk in $\Wcal$ whose initial vertex is $\ini(W(f_x^{i+1}))$, then \hyperref[clm:sameInitial=subwalk]{(ii)} implies that $e$ is on $W$.
\item $e \in E(\Qcal) \setminus E(\Qcal^\infty)$: $e$ is on some initial segment $Q f_x^\infty$ of a path $Q$ in $\Qcal$. More precisely, there is an integer $i$ and $x, x' \in S \cup I$, such that $e \in f_x^i Q f_{x'}^i$. Since $f_x^i \neq f_x^{i+1}$, similar to the previous case, there is a walk in $\Wcal$ containing $e$. \qedhere
\end{enumerate}
\end{proof}

An immediate corollary of the following is that any forbidden minor, of which there are infinitely many (\cite{Ing77}), for the class of finite gammoids is also a forbidden minor for infinite gammoids. 

\begin{thm}\label{thm:gammoidFiniteMinors}
Any finite-rank minor of a gammoid is also a gammoid. 
\end{thm}
\begin{proof}
The setting follows the first paragraph of the proof of \autoref{thm:gammoidOCMinors}. 
Suppose that $M$ has finite rank $r$. 
Since $R$ is coindependent, $V\- R$ is spanning in $N$. Therefore, $N / S$ also has rank $r$. 
Let $I\in M_L(D_1-S, B_1\- S)$, then $r = |B_0 \setminus T| = |B_1\setminus S| \geq |I|$ and, by \autoref{thm:linkableInShifted2linkableInDB0}(i), $I$ is in $\Ical(N / S)$. 
Conversely, if $I\in \Ical(N / S)$, then $I$ is finite. By \autoref{thm:linkableInMinor2linkableInShifted_finite}, $I$ is linkable in $(D_1 - S, B_1\- S)$. 
Hence $M_L(D_1 - S, B_1 \setminus S)$ is a strict gammoid presentation of $N / S$ and $M = M_L(D_1 - S, B_1 \setminus S) \setminus R$ is a gammoid.
\end{proof}

\subsection{Topological gammoids}
\label{sec:TopologicalGammoids}
A topological notion of linkability is introduced in \cite{Car14+}. Roughly speaking, a topological path from a vertex $v$ does not need to reach the exits as long as no finite vertex set avoiding that path can prevent an actual connection of $v$ to $B_0$.

Here we show that in fact, topological gammoids coincide with the finitary gammoids.  As a corollary, we see that topological gammoids are minor-closed.
The difference between a topological linkage and a linkage is that paths ending in the centre of a linking fan and spines of outgoing combs are allowed. Thus, to prove the following, it suffices to give a $\{C^O, F^\infty\}$-free dimaze presentation for the strict topological gammoid.

\begin{lem}\label{thm:topologicalGammoidIsGammoid}
Every strict topological gammoid is a strict gammoid.
\end{lem}

\begin{proof}
Let $(D', B_0')$ be a dimaze and $F$ be the set of all vertices that are the centre of a subdivision of $F^\infty$. Let $(D, B_0)$ be obtained from $(D', B_0')$ by deleting all edges whose initial vertex is in $F$ from $D'$ and $B_0 := B_0' \cup F$.

We claim that $M_{TL}(D, B_0) = M_{TL}(D', B_0')$.  
Let $\Pcal$ be a topological linkage of $I$ in $(D', B_0')$. 
Then the collection of the initial segments of each element of $\Pcal$ up to the first appearance of a vertex in $F$ forms a topological linkage of $I$ in $(D, B_0)$. 
Conversely, let $\Pcal$ be a topological linkage of $I$ in $(D, B_0)$. Note that any linkage in $(D, B_0)$ is a topological linkage in $(D', B_0')$. In particular the spikes of an outgoing comb whose spine $R$ is in $\Pcal$ form a topological linkage. Hence, $R$ is also the spine of an outgoing comb in $(D', B_0')$ by \autoref{thm:OR2OC}. So $I$ is topologically linkable in $(D, B_0)$.

Let $S \cup B_0$ be a base of $M_{TL}(D, B_0)$ and $\Qcal$ a set of disjoint spines of outgoing combs starting from $S$. 
We show that a set $I$ is topologically linkable in $(D, B_0)$ if and only if it is linkable in the $\Qcal$-shifted dimaze $(D_1, B_1)$.

Let $\Pcal$ be a topological linkage of $I$ in $(D, B_0)$. 
By \autoref{thm:maximalPQAWsDisjoint}, the set $\Wcal$ of maximal $\Pcal$-$\Qcal$-alternating walks starting in $I$ is a set of disjoint $\Qcal$-alternating walks possibly ending in $\ter(\Pcal) \cup S \subseteq B_1$. 
If there were an infinite walk, then it would have to start outside $S$ and give rise to a topologically linkable superset of $S\cup B_0$, by \autoref{thm:aws2paths} and \autoref{thm:OR2OC}. So each walk in $\Wcal$ is finite. By \autoref{thm:AW-path_shift}, $I$ is linkable in $(D_1, B_1)$.

Conversely let $I$ be linkable in $(D_1, B_1)$ and $\Wcal$ a set of disjoint finite $\Qcal$-alternating walks in $(D, B_0)$ from $I$ to $B_1$ provided by \autoref{thm:AW-path_shift}. By \autoref{thm:aws2paths}, $\Qcal \Delta \Wcal$ contains a set $\Rcal$ of disjoint paths or rays in $(D, B_0)$ from $I$ to $B_0$.
By \autoref{thm:ORinSymmetricDifference2OC}, any ray in $\Rcal$ is in fact the spine of an outgoing comb, so $I$ is topologically linkable in $(D, B_0)$.
\end{proof}

Now we can characterize strict topological gammoids among strict gammoids.
\begin{thm}\label{thm:characterizationFinitaryStrictGammoids}
The following are equivalent:
\begin{enumerate}
\item $M$ is a strict topological gammoid;
\item $M$ is a finitary strict gammoid;
\item $M$ is a strict gammoid such that any presentation is $\{C^O, F^\infty\}$-free; 
\item $M$ is a $\{ C^O, F^\infty \}$-free strict gammoid. 
\end{enumerate}
\end{thm}

\begin{proof}
$1. \Rightarrow 2.$ :
By \autoref{thm:finitarizationOfGammoid}, $M$ is a finitary matroid and by \autoref{thm:topologicalGammoidIsGammoid} it is a strict gammoid.

$2. \Rightarrow 3.$ :
Let $M_L(D, B_0)$ be any presentation of $M$.
Note that the union of any vertex $v \in V \setminus B_0$ and all the vertices in $B_0$ to which $v$ is linkable forms a circuit in $M$ (the fundamental circuit of $v$ and $B_0$). Suppose $(D, B_0)$ is not $\{ C^O, F^\infty \}$-free, then there is a vertex linkable to infinitely many vertices in $B_0$. But then $M$ contains an infinite circuit and is not finitary.

$3. \Rightarrow 4.$ : Trivial. 

$4. \Rightarrow 1.$ :
Take a $\{C^O, F^\infty\}$-free presentation of $M$. Then topological linkages coincide with linkages. Hence $M$ is a topological gammoid.
\end{proof}

Next we also characterize topological gammoids among gammoids. 

\begin{cor}\label{thm:characterizationFinitaryGammoids}
The following are equivalent:
\begin{enumerate}
\item $M$ is a topological gammoid;
\item $M$ is a finitary  gammoid;
\item $M$ is a $\{ C^O, F^\infty \}$-free gammoid.
\end{enumerate}
\end{cor}

\begin{proof}
$1. \Rightarrow 3.$ :
There exist a dimaze $(D, B_0)$ and $X \subseteq V$ such that $M = M_{TL}(D, B_0) \setminus X$. 
By \autoref{thm:characterizationFinitaryStrictGammoids}, there is a $\{C^O, F^\infty \}$-free dimaze $(D_1, B_1)$ such that $M_L(D_1, B_1) = M_{TL}(D, B_0)$. Hence, $M$ is a $\{C^O, F^\infty \}$-free gammoid.

$3. \Rightarrow 2.$ :
There exists a $\{ C^O, F^\infty \}$-free presentation of a strict gammoid $N$ of which $M$ is a restriction. By \autoref{thm:characterizationFinitaryStrictGammoids}, $N$ is finitary, thus, so is $M$.

$2. \Rightarrow 1.$ : There exist $(D, B_0)$ and $X\subseteq V$ such that $M = M_L(D, B_0) \- X$. 
Since $M\- X$ is finitary, $\Ccal( M\- X) = \Ccal( M^{\rm fin}\- X)$. By \autoref{thm:finitarizationOfGammoid}, the latter is equal to $\Ccal(M_{TL}(D, B_0)\- X)$. Hence, $M$ is a topological gammoid. 
\end{proof}

\begin{thm}
\label{thm:Top-minor-closed}
The class of finitary gammoids (or equivalently topological gammoids) is closed under taking minors.
\end{thm}
\begin{proof}
Let $M$ be a finitary gammoid. By \autoref{thm:characterizationFinitaryGammoids}, $M$ is a $\{ C^O, F^\infty\}$-free gammoid. Any minor of $M$ is a $C^O$-free gammoid by \autoref{thm:gammoidOCMinors}, and also $F^\infty$-free by \autoref{thm:noFanInShifted}. 
So any minor of $M$ is a finitary gammoid by \autoref{thm:characterizationFinitaryGammoids}.
\end{proof}

\section{Duality} \label{sec:duality}
While a finite strict gammoid is dual to a transversal matroid \cite{IP73}, an infinite strict gammoid need not be. So with the aim of understanding the dual of strict gammoids, we first introduce a natural extension of transversal matroids. This provides a description of the dual of $C^A$-free strict gammoids introduced in \cite{ALM}. However, it turns out that this extension does not contain the dual of all strict gammoids. 

Another result proved in \cite{IP73} states that the dual of a finite gammoid is a gammoid. 
This can be proved by observing that the dual of any finite strict gammoid is a transversal matroid and that finite gammoids are closed under contraction minors. 
The proof remains valid for those infinite gammoids that admit a presentation $(D,B_0)$ such that the underlying graph of $D$ is rayless. In general, the proof breaks down. For example, we shall see that $C^I$ defines a strict gammoid whose dual is not a transversal matroid.  
While this dual is a gammoid, a more badly behaved example exists:~there is a strict gammoid which is not dual to any gammoid. 

\subsection{Strict gammoids and path-transversal matroids} \label{sec:gammoidsTransversalDuality}
The class of path-transversal matroids is introduced as a superclass of transversal matroids, and proved to contain the dual matroids of any $C^A$-free strict gammoid. We shall see that an extra condition forces $C^A$-free strict gammoids to be dual to transversal matroids. 
On the other hand, even though path-transversal matroids extend transversal matroids, they do not capture the dual of all strict gammoids, as we shall see in \autoref{thm:infinitetree-legal}.

\medskip
Let us introduce a dual object of a dimaze.
Given a bipartite graph $G=(V,W)$, we call a matching $m_0$ onto $W$ an \emph{identity matching}, and the pair $(G,m_0)$ a \emph{bimaze}\footnote{Short for \emph{bi}partite \emph{maze}.}. 
We adjust two constructions of \cite{IP73} for our purposes. 

\begin{defi}\label{def:dualityConstruction}
Given a dimaze $(D,B_0)$, define a bipartite graph $D^\star_{B_0}$, with bipartition $(V, (V\- B_0)^\star)$, where $(V\- B_0)^\star: = \{v^\star : v\in V\- B_0\}$ is disjoint from $V$; and $E(D^\star_{B_0}):=m_0 \cup \{vu^\star : (u,v) \in E(D)\}$, where $m_0:=\{vv^\star : v\in V\setminus B_0 \}$. Call $(D,B_0)^\star:=(D^\star_{B_0},m_0)$ the 
\emph{converted bimaze} of $(D, B_0)$.
\end{defi}

Starting from a dimaze $(D, B_0)$, we write $(V\setminus B_0)^\star$, $m_0$ and $v^\star$ for the corresponding objects in \autoref{def:dualityConstruction}.

\begin{defi}\label{def:conDimaze}
Given a bimaze $(G,m_0)$, where $G=(V, W)$, define a digraph $G^\star_{m_0}$ such that $V(G^\star_{m_0}) := V$ and $E(G^\star_{m_0}) := \{(v,w) : wv^\star  \in E(G) \setminus m_0 \}$, where $v^\star$ is the vertex in $W$ that is matched by $m_0$ to $v \in V$. Let $B_0 : = V\- V(m_0)$. Call $(G,m_0)^\star:=(G^\star_{m_0},B_0)$ the \emph{converted dimaze} of $(G, m_0)$. 
\end{defi}

Starting from a bimaze $(G, m_0)$, we write $B_0$ and $v^\star$ for the corresponding objects in \autoref{def:conDimaze} and $(V\setminus B_0)^\star$ for the right vertex class of $G$.

\begin{figure}[htb]
 \centering
 \includegraphics[scale=0.8]{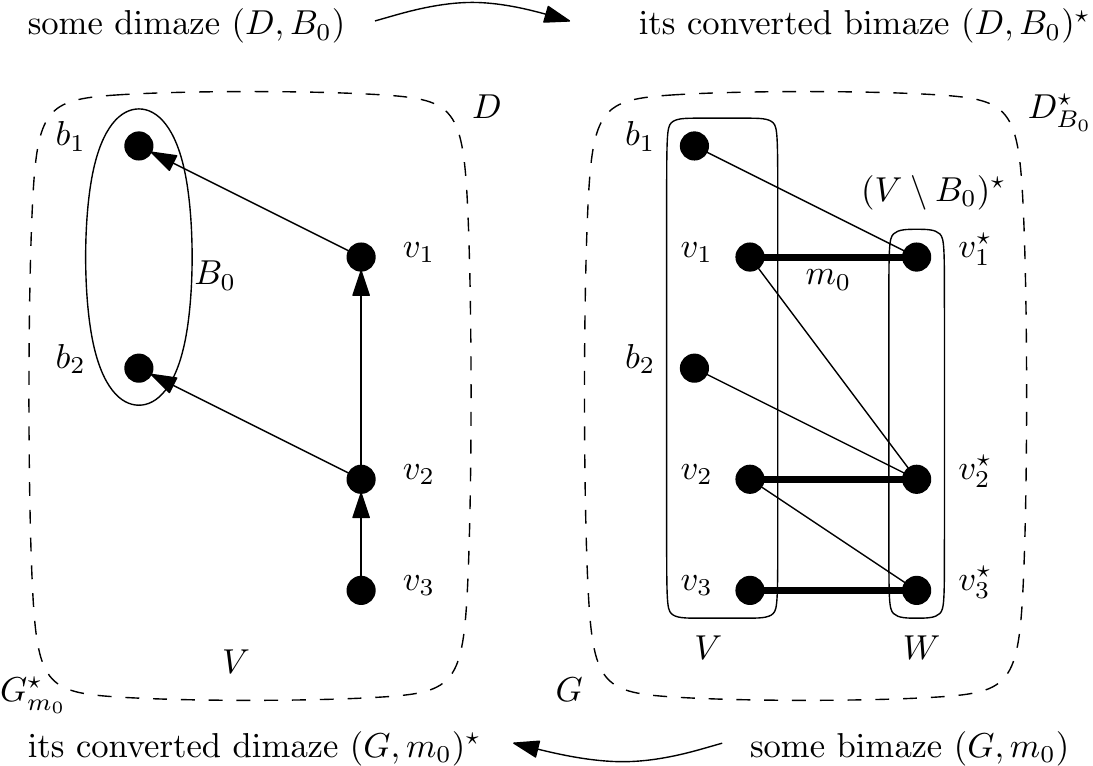}
 \caption{Converting a dimaze to a bimaze and vice versa}
 \label{fig:constructions}
\end{figure} 

Note that these constructions are inverse to each other (see \autoref{fig:constructions}). In particular,   
let $(G,m_0)$ be a bimaze, then 
\begin{equation}
\label{E:intdergraph}
(G,m_0)^{\star\star}=(G,m_0).
\end{equation}

Note that for any matching $m$, each component of $G[m_0\cup m]$ is either a path, a ray or a double ray. 
If $G[m_0\cup m]$ consists of only finite components, then $m$ is called a \emph{$m_0$-matching}. A set $I\subseteq V$ is \emph{$m_0$-matchable} if there is an $m_0$-matching of $I$. 

\begin{defi}
\label{def:path-transversal} 
Given a bimaze $(G,m_0)$, the pair of $V$ and the set of all $m_0$-matchable subsets of $V$ is denoted by $M_{PT}(G, m_0)$. If $M_{PT}(G,m_0)$ is a matroid, it is called a \emph{path-transversal matroid}.
\end{defi}

The correspondence between finite paths and $m_0$-matchings is depicted in the following lemma. 

\begin{lem}\label{thm:legalLinkageLemma}
Let $(D,B_0)$ be a dimaze. Then $B$ is linkable onto $B_0$ in $(D,B_0)$ iff $V\setminus B$ is $m_0$-matchable onto $(V\setminus B_0)^\star$ in $(D,B_0)^\star$.
\end{lem}
\begin{proof}
Suppose a linkage $\Pcal$ from $B$ onto $B_0$ is given. Let 
\[ m : = \{ vu^\star : (u,v) \in E(\Pcal) \} \cup \{ ww^\star : w \notin V(\Pcal) \}.\]
Note that $m$ is a matching from $V \setminus B$ onto $(V\setminus B_0)^\star$ in $D_{B_0}^\star$.
Any component induced by $m_0\cup m$ is finite, since any component which contains more than one edge  corresponds to a path in $\Pcal$. So $m$ is a required $m_0$-matching in $(D,B_0)^\star$.

Conversely let $m$ be an $m_0$-matching from $V \setminus B$ onto $(V\setminus B_0)^\star$. Define a linkage from $B$ onto $B_0$ as follows. 
From every vertex $v\in B$, start an $m_0$-$m$-alternating walk, which is finite because $m$ is an $m_0$-matching. Moreover, the walk cannot end with an $m_0$-edge because $m$ covers $(V\- B_0)^\star$. So the walk is either trivial or ends with an $m$-edge in $B_0$. 
As the $m$-edges on each walk correspond to a path from $B$ to $B_0$,  
together they give us a required linkage in $(D, B_0)$.
\end{proof}

\begin{pro}\label{thm:TmIsLtm}
Let $M_T(G)$ be a transversal matroid and $m_0$ a matching of a base $B$. Then $M_T(G)=M_{PT}(G,m_0)$. 
\end{pro}
\begin{proof}
Suppose $I\subseteq V$ admits a matching $m$. By the maximality of $B$, any infinite component of $m\cup m_0$ does intersect $V\setminus B$. Replacing the $m$-edges of all the infinite components by the $m_0$-edges gives an $m_0$-matching of $I$.
\end{proof}

In fact, the class of path-transversal matroids contains the class of transversal matroids as a proper subclass; see \autoref{thm:ARnotCotransversal} and \autoref{fig:ARIC}. Just as we can extend a linkage to cover the exits by trivial paths, any $m_0$-matching can be extended to cover $W$. 

\begin{lem} \label{thm:extendingToMaximal}
Let $(G,m_0)$ be a bimaze. For any $m_0$-matchable $I$, there is an $m_0$-matching from some $B\supseteq I$ onto $W$.
\end{lem}
\begin{proof}
Let $m$ be an $m_0$-matching of $I$. Take the union of all connected components of $m\cup m_0$ that meet $W-m$. The symmetric difference of $m$ and this union is a desired $m_0$-matching of a superset of $I$. 
\end{proof}

We find it convenient to abstract two properties of a dimaze and a bimaze.    
Given a dimaze $(D, B_0)$, let $(\dagger)$ be
\[I\in M_L(D,B_0) \mbox{ is maximal } \Leftrightarrow \exists \mbox{ linkage from } I \mbox{ onto } B_0  \tag{$\dagger$}.\]
Analogously, given a bimaze $(G,m_0)$, let $(\ddagger)$ be
\[I\in M_{PT}(G, m_0) \mbox{ is maximal } \Leftrightarrow \exists \mbox{ $m_0$-matching from } I \mbox{ onto } (V\setminus B_0)^\star.  \tag{$\ddagger$}\]
In some sense $(\dagger)$ and $(\ddagger)$ are dual to each other.

\begin{lem}\label{thm:daggerDuality}
A dimaze $(D,B_0)$ satisfies $(\dagger)$ iff $(D,B_0)^\star$ satisfies $(\ddagger)$. 
\end{lem}
\begin{proof}
Assume $(D, B_0)$ satisfies $(\dagger)$. To prove the backward direction of $(\ddagger)$, suppose there is an $m_0$-matching from $V\setminus B$ onto $(V\setminus B_0)^\star$. 
By \autoref{thm:legalLinkageLemma}, there is a linkage from $B$ onto $B_0$. Therefore, $B$ is maximal in $M_L(D, B_0)$ by $(\dagger)$. By \autoref{thm:extendingToMaximal}, any $m_0$-matchable superset of $V\setminus B$ may be extended to one, say $V\- I$, that is $m_0$-matchable onto $(V\- B_0)^\star$.  As before, $I\subseteq B$ is maximal in $M_L(D, B_0)$, so $I=B$ and hence, $V\- B$ is a maximal $m_0$-matchable set. To see the forward direction of $(\ddagger)$, suppose $V\-B$ is a maximal $m_0$-matchable set witnessed by an $m_0$-matching $m$,
that does not cover $v^\star \in (V\setminus B_0)^\star$. 
As $m$ is an $m_0$-matching, a maximal $m_0$-$m$-alternating walk starting from $v^\star$ ends at some vertex in $B$. So the symmetric difference of this walk and $m$ is an $m_0$-matching of a proper superset of $V\-B$ which is a contradiction.

Assume $(D,B_0)^\star$ satisfies $(\ddagger)$. 
The forward direction of $(\dagger)$ is trivial.
For the backward direction, 
suppose there is a linkage from $B$ onto $B_0$. Then there is an $m_0$-matching from $V\setminus B$ onto $(V\setminus B_0)^\star$ by \autoref{thm:legalLinkageLemma}. By $(\ddagger)$, $V\setminus B$ is maximal in $M_{PT}(D,B_0)^\star$. With an argument similar to the above, we can conclude that $B$ is maximal in $M_L(D, B_0)$.
\end{proof}

Now let us see how $(\dagger)$ helps to identify the dual of a strict gammoid. 

\begin{lem}
\label{thm:duality}
If a dimaze $(D,B_0)$ satisfies $(\dagger)$, then the dual of $M_L(D,B_0)$ is $M_{PT}(D,B_0)^\star$.
\end{lem}
\begin{proof}
By \autoref{thm:daggerDuality}, $(D,B_0)^\star$ satisfies $(\ddagger)$. 
Let $B$ be an independent set in $M_L(D, B_0)$. Then $B$ is maximal if and only if there is a linkage from $B$ onto $B_0$. 
By \autoref{thm:legalLinkageLemma}, this holds if and only if there is an $m_0$-matching from $V\setminus B$ onto $(V\setminus B_0)^\star$, which by $(\ddagger)$ is equivalent to $V\setminus B$ being maximal in $M_{PT}(D,B_0)^\star$.

To complete the proof, it remains to see that every $m_0$-matchable set can be extended to a maximal one, which follows from  \autoref{thm:extendingToMaximal} and $(\ddagger)$.
\end{proof}

Note that while we do not need it,  the twin of \autoref{thm:duality} is true, namely, 
if a bimaze $(G, m_0)$ satisfies $(\ddagger)$, then $M_{PT}(G, m_0)$ is a matroid dual to $M_L(G,m_0)^\star$.

\medskip
To summarize, the dual of strict gammoids examined in \autoref{thm:ACfree} is given as follows.

\begin{thm}\label{thm:legalDuality}
(i) Given a $C^A$-free dimaze $(D, B_0)$, $M_L(D,B_0)$ is a matroid dual to $M_{PT}(D,B_0)^\star$.

(ii) Given a bimaze $(G,m_0)$, if $(G,m_0)^\star$ is $C^A$-free, then $M_{PT}(G,m_0)$ is a matroid dual to $M_L(G,m_0)^\star$.
\end{thm}

\begin{proof}
(i) This is the direct consequence of \autoref{thm:ACfree} and \autoref{thm:duality}.

(ii) Apply part (i) and \eqref{E:intdergraph}.
\end{proof}

One might hope that, in the first part of the theorem, the path-transversal matroid $M_{PT}(D, B_0)^\star$ is in fact the transversal matroid $M_T(D, B_0)^\star$.
However, the dimaze $R^I$ defines a strict gammoid whose dual is not the transversal matroid defined by the converted bimaze.
It turns out that $R^I$ is the only obstruction to this hope.

\begin{thm}
\label{thm:ARIRduality}
(i) Given an $\{R^I,C^A\}$-free dimaze $(D,B_0)$, $M_L(D,B_0)$ is a matroid dual to $M_T(D^\star_{B_0})$. 

(ii) Given a bimaze $(G,m_0)$, if $(G,m_0)^\star$ is $\{R^I,C^A\}$-free, then $M_T(G)$ is a matroid dual to $M_L(G,m_0)^\star$.
\end{thm}
\begin{proof}
(i) This follows from \autoref{thm:legalDuality}(i) and the fact that for
an $R^I$-free dimaze $(D,B_0)$, we have $M_T(D^\star_{B_0})=M_{PT}(D,B_0)^\star$.
The proof of the latter is similar to the one given to \autoref{thm:TmIsLtm} and omitted.

(ii) Apply part (i) and \eqref{E:intdergraph}. 
\end{proof}

It appears that $C^A$ is a natural constraint in the above theorem. 
\begin{exm}
\label{thm:ARnotCotransversal}
The strict gammoid defined by the dimaze $C^A$ (\autoref{fig:ARIC}a) is not cotransversal.
\end{exm}
\begin{proof}
Since $V\setminus B_0 + v$ is a base for every $v \in B_0$, $B_0$ is an infinite cocircuit. On the other hand, every vertex $v$ of $B_0$ is contained in a finite cocircuit, namely $v$ and its in-neighbours. 
So by \autoref{thm:TransversalInfiniteCircuit}, the dual is not transversal.
\end{proof}

Here is a question which is in some sense converse to \autoref{thm:ARIRduality}(i).

\begin{que}
\label{que:cotransversalACIR}
Is every cotransversal strict gammoid $\{C^A, R^I\}$-free? 
\end{que}

Although the class of path-transversal matroids contains that of transversal matroids properly, not every strict gammoid has its dual of this type. To show this, we first note that in a path-transversal matroid $M_{PT}(G, m_0)$, if $C$ is the fundamental circuit of $u$, then 
$N(C) = m_0(C-u)$. Indeed, $N(u) \subseteq m_0(C-u)$; and for any $v\in C-u$, since there is an $m_0$-alternating path from $u$ ending in $v$, $v$ cannot have any neighbour outside $m_0(C-u)$. 

\begin{exm}\label{thm:infinitetree-legal} 
Let $T$ be a rooted tree such that each vertex has infinitely many children, with edges directed towards $B_0$, which consists of the root and vertices on alternating levels. Then $M_L(T,B_0)$ is a strict gammoid that is not dual to any path-transversal matroid.
\end{exm}

\begin{proof}
In \cite[Theorem 3.4]{ALM}, it was proved that $M:=M_L(T, B_0)$ is a matroid. Suppose that $M^*=M_{PT}(G, m')$. Let $\Qcal$ be a linkage of $B:=V - m'$ to $B_0$. 
Since $(T, B_0)$ is $C^O$-free, by \autoref{thm:S_shiftedToExits}, we have $M= M_L(D_1, B_1)$ where $(D_1, B_1)$ is the $\Qcal$-shifted dimaze. 
By construction, the underlying graph of $D_1$ is also a tree. 

By \cite[Corollary 3.6]{ALM}, $(D_1, B_1)$ contains a subdivision $R$ of $C^A$. Let $\{s_i: i \geq 1\} :=R\cap B_1$ and  
$U:=\{ u_i: i \geq 1\}$ be the set of vertices of out-degree 2 on $R$ such that $u_i$ is joined to $s_i$ and $s_{i+1}$ in $R$. 
Let 
$U_i := \{ v\in T: v\mbox{ is separated from $R\cap B_1$ by $u_i$}\}$ and 
$S_i := \{ v\in T: v$ can be linked to $s_i$ in $D_1\- U\}$.

Since $D_1$ is a tree, $\{ U_i, S_i: i\geq 1 \}$ is a collection of pairwise disjoint sets. 

Let $C := \bigcup_{i\geq 1} S_i$. Any linkable set in $V\- C$ has a linkage that misses an exit in $R\cap B_1$. Since $D_1$ is a tree, $(B_1- R)\cup U + c$ for any $c\in C$ is a base of $M$. Hence, $C$ is a circuit in $M^*$. For a contradiction, we construct an $m'$-matching of $C$ in $(G, m')$. 

In $M^*$, the fundamental circuit of $s_i$ with respect to $B_1$ is $S_i\cup U_{i-1}\cup U_i$ (with $U_0 = \emptyset$). By the remark before the example,  $N(S_i\cup U_{i-1}\cup U_i) = m'(S_i\cup U_{i-1}\cup U_i - s_i)$ for $i \geq 1$. 

We claim that for $i\geq 1$, in any $m'$-matching $m$ of $\bigcup_{j\leq i} S_j$, the maximal $m$-$m'$-alternating walk from $s_j$ ends in $m'(U_j)$ for $j\leq i$. Note that such a walk cannot end in $m'(S_j)$ as those vertices are incident with $m$-edges. 
Since $N(S_1)\subseteq m'(S_1 \cup U_1)$, the claim is true for $i=1$. Assume that it is true for $i-1$. Consider an $m'$-matching $m_1$ of $\bigcup_{j\leq i} S_j$. Let $P_j$ be the maximal $m_1$-$m'$-alternating walk starting from $s_j$. By assumption, $P_j$ ends in $m'(U_j)$ for each $j<i$. As $P_i$ ends in $m'(U_{i-1}\cup U_i)$, we are done unless it ends in $m'(U_{i-1})$. In that case, the union of an $m'$-matching of $C\- \bigcup_{j\leq i} S_j$ with  
\[ (m'  \upharpoonright {\bigcup_{j\leq i} S_j}) \Delta \bigcup_{j\leq i} E(P_j)  \]
is an $m'$-matching of $C$, a contradiction. 

Therefore, there is a collection of pairwise disjoint $m'$-alternating walks $\{ P'_i : i\geq 1\}$ where $P'_i$ starts from $s_i$ and ends in $m'(U_i)$. Then $m' \Delta \bigcup_{i\geq 1} E(P'_i)$ is an $m'$-matching of $C$, a contradiction which completes the proof. 
\end{proof}

Here is a question similar to \autoref{que:cotransversalACIR} akin to \autoref{thm:legalDuality}.

\begin{que}
\label{que:ACclt}
Is every strict gammoid which is dual to a path-transversal matroid $C^A$-free?
\end{que} 

It may be interesting to investigate path-transversal systems further. For example, while a path-transversal system need not satisfy (IM), it might be the case that (I3) always holds. 
\begin{con}
\label{con:legal_I3}
Any path-transversal system satisfies (I3). 
\end{con}

\begin{figure}
 \centering
 \includegraphics[width = 8cm]{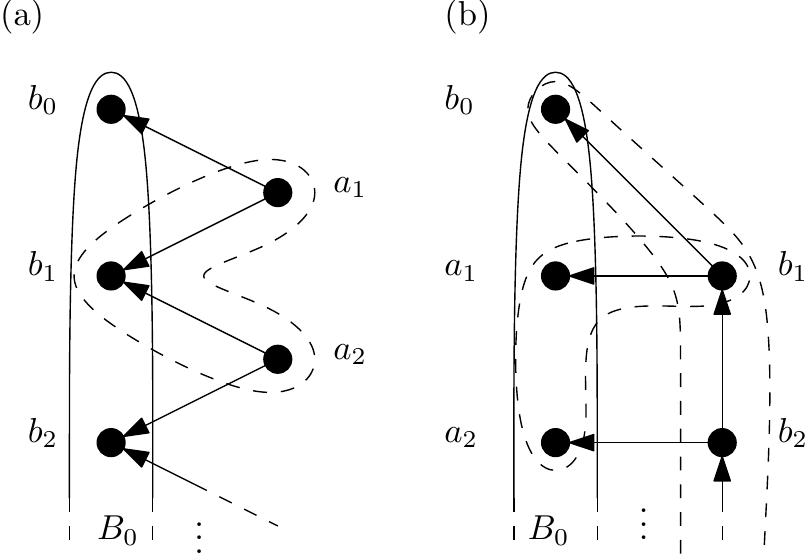}
 \caption{An alternating ray and an isomorphic incoming comb}
 \label{fig:ARIC}
 \end{figure}

\subsection{Finitary transversal matroids}\label{sec:NotCoStrictGammoid}
Our aim in this section is to give a transversal matroid that is not dual to any strict gammoid. To this end, we extend some results in \cite{Bon72} and \cite{BW71}. The following identifies edges that may be added to a presentation of a finitary transversal matroid without changing the matroid. 
\begin{lem}\label{thm:finite-transversal-non-edge}
Suppose that $M_T(G)$ is finitary. Let $K$ be a subset of $\{vw\notin E(G) : v\in V, w\in W\}$. Then the following are equivalent:
\begin{enumerate}
\item $M_T(G)\neq M_T(G+K)$;
\item there are $vw \in K$ and a circuit $C$ with $v\in C$ and $w\notin N(C)$;
\item there is $vw\in K$ such that $v$ is not a coloop of $M_T(G)\- N(w)$.
\end{enumerate} 
\end{lem}
\begin{proof}
1.~holds if and only if there is a circuit $C$ in $M_T(G)$ which is matchable in $G+K$. This, since $C$ is finite, in turn holds if and only if there is $v\in C$ that can be matched outside $N(C)$ in $G+K$, i.e.~2. holds. 

The equivalence between 2.~and 3.~is clear since a vertex is not a coloop if and only if it lies in a circuit.
\end{proof} 

Given a bipartite graph $G$, recall that a presentation of a transversal matroid $M$ as $M_T(G)$ is maximal if $M_T(G+vw)\neq M_T(G)$ for any $vw\notin E(G)$ with $v\in V, w\in W$. Thus, the previous lemma implies that if $M_T(G)$ is finitary, then $G$ is maximal if and only if $M\- N(w)$ is coloop-free for any $w\in W$.
Bondy \cite{Bon72} asserted that there is a unique maximal presentation for any finite coloop-free transversal matroid; where two presentations of a transversal matroid by bipartite graphs $G$ and $H$ are \emph{isomorphic} if there is a graph isomorphism from $G$ to $H$ fixing the left vertex class pointwise. 
 
\begin{pro}
\label{thm:transversalMax}
Every finitary transversal matroid $M$ has a
\! unique maximal presentation. 
\end{pro}
\begin{proof}
Let $M=M_T(G)$. 
Adding all $vw$ with the property that there is not any circuit $C$ with $v\in C$ and $w\notin N(C)$ gives a maximal presentation of $M$ by \autoref{thm:finite-transversal-non-edge}. In particular, any coloop is always adjacent to every vertex in $W$. So without loss of generality, we assume that $M$ is coloop-free. 

Now let $G$ and $H$ be distinct maximal presentations of $M$.

{\bf Claim 1.} For any finite subset $F$ of $V$, the induced subgraphs $G[F\cup N_G(F)]$ and $H[F\cup N_H(F)]$ are isomorphic.

For every $v\in F$ pick a circuit $C_v$ with $v\in C_v$. By \autoref{thm:finite-transversal-non-edge}, for every $vw\in \{ xy\notin E(G) : x\in F, y\in N_G(F)\}$, there is a circuit $C_{vw}$ with $v\in C$ and $w\notin N_G(C)$. Let $F_G$ be the union of all $C_v$'s and $C_{vw}$'s. Analogously define $F_H$ and let $F'=F_G\cup F_H$. Extend the presentations $G[F'\cup N_G(F')]$ and $H[F'\cup N_H(F')]$ of $M|F'$ to maximal ones $G'$ and $H'$ respectively, between which there is a graph isomorphism fixing the left vertex class pointwise by Bondy's result. Restricting the isomorphism to $F\cup N_G(F)$ is an isomorphism of $G[F\cup N_G(F)]$ and $H[F\cup N_H(F)]$, as by definition of $F'$ and \autoref{thm:finite-transversal-non-edge}, no non-edge between $F$ and $N_G(F)$ is an edge in $G'$ (analogously between $F$ and $N_H(F)$ in $H'$).

\medskip
Without loss of generality, there is an $A\subseteq V$ such that $g:=|\{w\in W(G): N_G(w)=A\}| < |\{w\in W(H): N_H(w) = A\}|=:h$. Note that as $H$ is a maximal presentation, by \autoref{thm:finite-transversal-non-edge}, $M\-A$ is coloop-free.

As $M$ is coloop-free, so is $M.A$. 
Let $B_1$ be a base of $M.A$ and extend $B_1$ to a base of $M$ which admits a matching $m$; thus $m$ contains a matching of a base of $M\- A$. 
Since $M\- A$ is coloop-free, by \autoref{thm:coversRHS}, the neighbourhood of each vertex matched by $m$ to a vertex in $B_1$ is a subset of $A$. 
Thus, $M.A$ can be presented with the subgraphs induced by $A\cup \{w\in W: N(w)\subseteq A\}$ in both graphs; 
call these subgraphs $G_1$ and $H_1$. For any $w\in W(G_1)$, since $M\- N_{G_1}(w)$ is coloop-free, so is $M.A\- N_{G_1}(w)$.  By \autoref{thm:finite-transversal-non-edge}, $G_1$ (analogously $H_1$) is a maximal presentation of $M.A$.

{\bf Claim 2.} Given a family $(N_j)_{j\in J}$ of finite subsets of $W$, if the intersection of any finite subfamily has size at least $k$, then the intersection of the family has size at least $k$. 

Let $N = \bigcap_{j\in J} N_j$. Suppose $|N| < k$. Fix some $j_0\in J$ and for each element $y\in N_{j_0}\- N$ pick some $N_y$ such that $y\notin N_y$. Then $|N_{j_0} \cap \bigcap_{y\in N_{j_0} \- N} N_y|= |N| < k$, which is a contradiction. 

\medskip
By Claim 2, there is a finite set $F\subseteq A$ such that $|\bigcap_{v\in F} N_{G_1}(v)|=g$. But Claim 1 says that $F$ has at least $h>g$ common neighbours in $H_1$; this contradiction completes the proof. 
\end{proof}

In \cite{ALM14}, it is proved that a cofinitary strict gammoid always admits a presentation that is $\{R^I, C^A\}$-free. 
To show that the following finitary transversal matroid is not dual to a strict gammoid, it suffices to show that there is no bimaze presentation whose converted dimaze is $C^A$-free. 

\begin{exm}
\label{thm:transversalNotDSG}
Define a bipartite graph $G$ as $V(G) =\{v_i, A_i: i\geq 1\}$ and $E(G) = \{ v_1A_1, v_2A_1, v_1A_3, v_2A_3\} \cup \{v_{2i-3}A_i, v_{2i-2}A_i, v_{2i-1}A_i, v_{2i}A_i: i\geq 2\}$. Then $M=M_T(G)$ is not dual to a strict gammoid. 
\end{exm}
\begin{proof}
As $G$ is left locally finite, $M$ is a finitary matroid. Assume for a contradiction that $M^* = M_L(D, B_0)$. By a characterization of cofinitary strict gammoids in \cite{ALM14}, we may assume that $(D, B_0)$ is $\{R^I, C^A\}$-free. Then by \autoref{thm:ARIRduality}, $M = M_T(D, B_0)^\star$. 

Now it can be checked that all $M\- N(w_i)$ are coloop-free. 
By \autoref{thm:finite-transversal-non-edge}, $G$ is the maximal presentation of $M$. The same lemma also implies that any minimal presentation $G'$ is obtained by deleting edges from $\{v_1A_3, v_2A_3\}$ and at most one from $\{v_1A_2, v_2A_2\}$. In particular, all presentations of $M$ differ from $G$ only finitely. It is not difficult to check that with any matching $m_0$ of a base, $(G, m_0)^\star$ contains a subdivision of $C^A$. Hence, there is no bimaze presentation of $M$ such that the converted dimaze is $C^A$-free, contradicting that $(D, B_0)^\star$ is such a presentation. 
\end{proof}

We remark that the above transversal matroid is dual to a gammoid, see \autoref{fig:TMnotDualTo_SG}. However, in the next section, we give a transversal matroid that is not dual to any gammoid. 

\begin{figure}
\centering
 \includegraphics[scale=0.9]{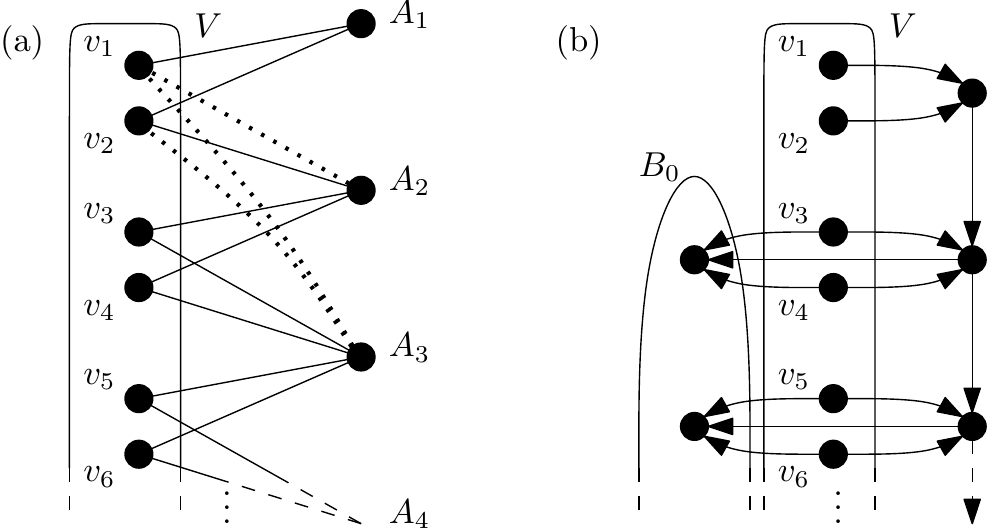}
 \caption{A transversal matroid which is not dual to a strict gammoid and a gammoid presentation of its dual}
 \label{fig:TMnotDualTo_SG}
\end{figure}

\subsection{Infinite tree and gammoid duality}\label{sec:discussionDuality}
To show that there is a strict gammoid not dual to a gammoid, we prove the following lemmas, whose common setting is that a given dimaze $(D, B_0)$ defines a matroid $M_L(D, B_0)$. For a linkage $\Qcal$ and any $X\subseteq \Ini(\Qcal)$, $\Qcal \upharpoonright X:=\{Q\in \Qcal : \Ini(Q)\in X \}$; when $X=\{x\}$, we write simply $Q_x$. 

\begin{lem}\label{thm:circuitAWexistence}
 Let $b$ be an element in an infinite circuit $C$, $\Qcal$ a linkage from $C-b$. Then $b$ can reach infinitely many vertices in $C$ via $\Qcal$-alternating walks. 
\end{lem}
\begin{proof}
Given any $x\in C-b$, let $\Pcal$ be a linkage of $C-x$. Let $W$ be a maximal $\Pcal$-$\Qcal$-alternating walk starting from $b$. If $W$ is infinite, then we are done. Otherwise, $W$ ends in either $\Ter(\Pcal) \- \Ter(\Qcal)$ or $\Ini(\Qcal) \- \Ini(\Pcal) = \{x\}$. The former case does not occur, since it gives rise to a linkage of $C$ by \autoref{thm:AW2Linkage}(i), contradicting $C$ being a circuit. As $x$ was arbitrary, the proof is complete. 
\end{proof}

\begin{lem}\label{thm:AWs-NotMeetTwo}
For $i=1, 2$, let $C_i$ be a circuit of $M$, $x_i, b_i$ distinct elements in $C_i \- C_{3-i}$. 
Suppose that $(C_1\cup C_2)\-\{b_1,b_2\}$ admits a linkage $\Qcal$. Then any two $\Qcal$-alternating walks $W_i$ from $b_i$ to $x_i$, for $i=1, 2$, are disjoint.
\end{lem}

\begin{proof}
Suppose that $W_1 = w^1_0 e^1_0 w^1_1 \ldots w^1_n$ and $W_2 = w^2_0 e^2_0 w^2_1 \ldots w^2_m$ are not disjoint. Then there exists a first vertex $v = w^1_j$ on $W_1$ such that $v = w^2_k \in W_2$ and either $v \in V(\Qcal)$ and $e^1_j = e^2_k \in E(\Qcal)$ or $v \notin V(\Qcal)$.
In both cases $W_3 := W_1 v W_2$ is a $\Qcal$-alternating walk from $b_1$ to $x_2$. 
Let $v'$ be the first vertex of $W_3$ in $V(\Qcal\upharpoonright (C_2 - b_2) \setminus C_1)$ and $Q$ the path in $\Qcal$ containing~$v'$.
Then $W_3 v' Q$ is a $(\Qcal\upharpoonright (C_1 - b_1))$-alternating walk from $b_1$ to $B_0\-\Ter(\Qcal\upharpoonright (C_1 - b_1))$, which by \autoref{thm:AW2Linkage}(i) contradicts the dependence of $C_1$. Hence $W_1$ and $W_2$ are disjoint.
\end{proof}

\begin{lem}\label{thm:AWs-NotMeetInf}
Let $\{C_i : i\in N\}$ be a set of circuits of $M$; $x_i, b_i$ distinct elements in $C_i\-\bigcup_{j\neq i} C_j$. 
Suppose that $\bigcup_{i\in N} C_i\-\{b_i : i\in N\}$ admits a linkage $\Qcal$. Let $W_i$ be a $\Qcal$-alternating walk from $b_i$ to $x_i$. If $X\subseteq V$ is a finite set containing $C_i\cap C_j$ for any distinct $i, j$, then only finitely many of $W_i$ meet $\Qcal\upharpoonright X$.
\end{lem}
\begin{proof}
By \autoref{thm:AWs-NotMeetTwo}, the walks $W_i$ are pairwise disjoint. Since $\Qcal\upharpoonright X$ is finite, it can be met by only finitely many $W_i$'s. 
\end{proof} 

We are now ready to give a counterexample to classical duality in gammoids.

\begin{exm}\label{thm:InfiniteTreeNotCogammoid}
Let $(T,B_0)$ be the dimaze defined in \autoref{thm:infinitetree-legal}.
The dual of the strict gammoid $M=M_L(T,B_0)$ is not a gammoid.
\end{exm}
\begin{proof}
Suppose that $M^* = M_L(D,B_1)\upharpoonright V$, where $V:=V(T)$. Fix a linkage $\Qcal$ of $V\- B_0$ in $(D, B_1)$. 
For $b\in B_0$, let $C_b$ be the fundamental cocircuit of $M$ with respect to $B_0$. 
Then for any (undirected) ray $b_0 x_0 b_1 x_1 \cdots $in $T$, $C:=\bigcup_{k\in \mathbb N} C_{b_k} \- \{x_k: k\in\mathbb N\}$ is a cocircuit of $M$. We get a contradiction by building a linkage for $C$ in $(D, B_1)$ inductively using disjoint $\Qcal$-alternating walks. 

Let $b_0$ be the root of $T$. By \autoref{thm:circuitAWexistence}, there is a $\Qcal$-alternating walk $W_0$ from $b_0$ to one of its children $x_0$. 
At step $k>0$, from each child $b$ of $x_{k-1}$ in $T$, by \autoref{thm:circuitAWexistence}, there is a $\Qcal$-alternating walk $W_b$ in $(D,B_1)$ to a child $x$ of $b$. 
Applying \autoref{thm:AWs-NotMeetInf} on $\{C_i: i\in N^-(x_{k-1})-b_{k-1}\}$ with $X = \{x_{k-1}\}$, we may choose $b_k:=b$, $x_k:=x$ such that $W_k:= W_b$ avoids $Q_{x_{k-1}}$. 

By \autoref{thm:AWs-NotMeetTwo}, distinct $W_k$ and $W_{k'}$ are disjoint. 
Moreover, as each $W_k$ avoids $Q_{x_{k-1}}$, \autoref{thm:AW2Linkage}(i) implies that $W_k$ can only meet $\Qcal$ at $Q_x$ where $x\in C_{b_k}- x_{k-1}$. Then $E(\Qcal) \triangle \bigcup_{k \in \mathbb{N}} E(W_k)$ contains a linkage of $C$.
\end{proof}

By ignoring the direction of the edges of $T$ and adding a leaf to each vertex in $B_0$, we can present $M_L(T, B_0)$ as a transversal matroid \cite{ALM}. Thus, not every transversal matroid is dual to a gammoid. 
It is also possible to prove that any $R^A$-free strict gammoid is dual to a gammoid. In fact, we expect more. 

\begin{con}\label{con:closedDuality}
The class of $C^A$-free gammoids is closed under duality.
\end{con}

\section{Open problems}
\label{sec:problems}
For easy reference for open problems, we list the following: Questions \ref{que:cotransversalACIR} and \ref{que:ACclt}, Conjectures \ref{con:legal_I3} and \ref{con:closedDuality}. A general closure problem is to find a reasonable superclass of gammoids that is closed under duality and taking minors. A possible first step is to identify the duals of strict topological gammoids, which is an interesting problem on its own. 

In \autoref{sec:duality}, we saw that a finitary transversal matroid has a unique maximal presentation. On the other hand, minimal presentations \cite{BW71} are presentations in which the removal of any edge changes the transversal system. It is easy to see that any finitary transversal matroid has a minimal presentation.  
Indeed, let $M=M_T(G)$ be a finitary transversal matroid where $G=(V,W)$ is left locally finite. Then for every finite subset $F$ of $V$, there are finitely many minimal presentations of $M \upharpoonright F$. By compactness, there is a minimal presentation for the whole of $M$. However, it is open whether any infinite transversal matroid has one.
\begin{que}
Does every transversal matroid admit a minimal presentation? 
\end{que}

\end{document}